\documentclass{amsart}
\usepackage[utf8]{inputenc}

\usepackage{amsmath}
\usepackage{amssymb}
\usepackage{verbatim}
\usepackage{geometry}

\usepackage[colorlinks=true,urlcolor=blue,citecolor=red,linkcolor=blue,linktocpage,pdfpagelabels,bookmarksnumbered,bookmarksopen]{hyperref}

\usepackage[hyperpageref]{backref}
\usepackage{soul}


\date{\today}

\newtheorem{thm}{Theorem}[section]
\newtheorem{cor}[thm]{Corollary}
\newtheorem{prop}[thm]{Proposition}
\newtheorem{lemma}[thm]{Lemma}
\theoremstyle{definition}
\newtheorem{defn}[thm]{Definition}

\newtheorem{rem}[thm]{Remark}
\newtheorem*{ack}{Acknowledgments}

\numberwithin{equation}{section}

\title[Comparison principle and geometric estimates]{A comparison principle for the Lane-Emden equation\\ and applications to geometric estimates}

\author[Brasco]{Lorenzo Brasco}
\address[L.\ Brasco]{Dipartimento di Matematica e Informatica
	\newline\indent
	Universit\`a degli Studi di Ferrara
	\newline\indent
	Via Machiavelli 35, 44121 Ferrara, Italy}
\email{lorenzo.brasco@unife.it}

\author[Prinari]{Francesca Prinari}

\address[F. Prinari]{Dipartimento di Scienze Agrarie, Alimentari e Agro-ambientali
\newline\indent 
Universit\`a di Pisa
\newline\indent
Via del Borghetto 80, 56124 Pisa, Italy}
\email{francesca.prinari@unipi.it}

\author[Zagati]{Anna Chiara Zagati}
\address[A.\,C.\ Zagati]{Dipartimento di Scienze Matematiche, Fisiche e Informatiche
	\newline\indent
	Universit\`a di Parma
	\newline\indent
	Parco Area delle Scienze 53/a, Campus, 43124 Parma, Italy}
	\email{annachiara.zagati@unipr.it}

\begin{document}

\begin{abstract}
We prove a comparison principle for positive supersolutions and subsolutions to the Lane-Emden equation for the $p-$Laplacian, with subhomogeneous power in the right-hand side. The proof uses variational tools and the result applies with no regularity assumptions, both on the set and the functions.
We then show that such a comparison principle can be applied to prove: uniqueness of solutions; sharp pointwise estimates for positive solutions in convex sets; localization estimates for maximum points and sharp geometric estimates for generalized principal frequencies in convex sets. 
\end{abstract}

\maketitle

\begin{center}
\begin{minipage}{11cm}
\small
\tableofcontents
\end{minipage}
\end{center}

\section{Introduction}

\subsection{The Lane-Emden equation}

Let $1<p<\infty$, we will indicate by
\[
\Delta_p u=\mathrm{div\,}(|\nabla u|^{p-2}\,\nabla u),
\]
the $p-$Laplace operator. Given $\Omega\subset\mathbb{R}^N$ an open set, in this paper we deal with weak solutions to the following quasilinear equation 
\begin{equation}\label{LE}
-\Delta_{p} u = \alpha\, |u|^{q-2}\,u, \qquad \mbox{ in } \Omega,
\end{equation} 
where $\alpha>0$ is a given constant. We will focus on the {\it sub-homogeneous case}, i.e. we will always consider the case
\[
1<q<p.
\] 
The equation \eqref{LE} will be coupled with a Dirichlet boundary condition, not necessarily identically to zero. As for the set $\Omega$, we will make no regularity assumptions on its boundary and will allow for a large class of sets, possibly unbounded or even with infinite volume (we refer to Definition \ref{defn:qadmissible} below for the precise condition). 
\par
We recall that the equation \eqref{LE} is connected to the stationary solutions of the following {\it doubly nonlinear slow diffusion} equation
\begin{equation}
\label{DNE}
\partial_t (|u|^{q-2}\,u)=\Delta_p u,\qquad \mbox{ in } (0,+\infty)\times \Omega.
\end{equation}
More precisely, by looking for nontrivial solutions of the type $u(t,x)=T(t)\,X(x)$, it is straightforward to see that the spatial part $X$ must solve precisely \eqref{LE},
%\[
%-\Delta_p X=\alpha\,|X|^{q-2}\,X,\qquad \mbox{ in }\Omega,
%\]
while the temporal part must decay polynomially to $0$, i.\,e. we have
\[
T(t)\sim t^{-\frac{1}{p-q}},\qquad \mbox{ for } t\to +\infty.
\]
As simple as it is, this formal computation is the first important step in the understanding of the long-time behaviour of solutions to \eqref{DNE}, together with the identification of their limit profiles. 
\par
For ease of completeness, we recall that this has been made rigourous in \cite{SV}, at least in the case of {\it positive solutions} and for homogenenous Dirichlet boundary conditions. More precisely, \cite[Theorem 2.1]{SV} proves\footnote{We notice that \cite{SV} deals with the (apparently) different equation
\[
\partial_t v=\Delta_p (|v|^{m-1}\,v),
\]
under the restriction $(p-1)\,m>1$.
However, by taking the nonlinear scaling $u=|v|^{m-1}\,v$ we formally end up with \eqref{DNE}, where $q=1+1/m$. Observe that we have
\[
(p-1)\,m>1 \qquad \Longleftrightarrow \qquad q=1+\frac{1}{m}<p,
\]
which is exactly the range we are interested in.} that 
\[
\lim_{t\to+\infty}\left\|t^\frac{1}{p-q}\,u(t,\cdot)- w_{\Omega,\alpha}(\cdot)\right\|_{L^\infty(\Omega)}=0,
\]
provided $\Omega\subset\mathbb{R}^N$ is an open bounded set, with sufficiently smooth boundary.
Here $w_{\Omega,\alpha}$ is the unique positive solution of \eqref{LE} vanishing at the boundary, with the choice
\[
\alpha=\frac{q-1}{p-q}.
\]
This is the generalization of a classical result by Aronson and Peletier (see \cite{AP} and also \cite{Va} for a simpler proof) for the case $p=2$ and $1<q<2$, where the equation \eqref{DNE} reduces to the so-called {\it porous medium equation}. 
\par
However, it seems that a complete picture about the long-time behaviour of solutions to \eqref{DNE} is still missing, in the case of sign-changing solutions and for more general boundary data. We just cite the recent paper \cite{BV} for the case of sign-changing solutions, in a particular situation. We also refer to
 \cite[Section 1.3]{BDL} and the references therein, for a discussion on the physical relevance of equation \eqref{DNE}.

\subsection{Main results}
Let us now go back to the equation \eqref{LE}.
The first result of this paper is a comparison principle for positive subsolutions and supersolutions of this equation, see Theorem \ref{thm:comparison} below. In proving such a result, we will take a purely variational point of view. Indeed, the cornerstone of the proof will be the so-called {\it hidden convexity principle} for the $p-$Dirichlet integral, i.\,e. the fact that 
\[
\psi\mapsto \int_\Omega |\nabla \psi|^p\,dx
\]
is convex, on the cone of non-negative functions, along curves of the form
\[
t\mapsto \Big((1-t)\,\psi_0^q+t\,\psi_1^q\Big)^\frac{1}{q},\qquad
\]
for $1\le q\le p$. This is the generalization of a remarkable result by Benguria, originally obtained for $p=q=2$ in \cite{Be1} (see also \cite[Lemma 4]{BBL}) and first generalized to the case $p=q\not =2$ by D\'iaz and Sa\'a in \cite[Lemme 1]{DS}. We refer to Remark \ref{rem:HC} for more bibliographical details.
\par
This property in particular entails that the energy functional associated to \eqref{LE}, i.e. 
\[
\mathfrak{F}_{q,\alpha}(\psi)=\frac{1}{p}\,\int_\Omega |\nabla\psi|^p\,dx-\frac{\alpha}{q}\,\int_\Omega |\psi|^q\,dx,
\]
is actually convex, in this suitable sense, despite the presence of the concave lower-order perturbation given by the term with $|\psi|^q$.
This simple observation permits to recover the comparison principle for our equation, by adapting the classical proof of the comparison principle for quasilinear elliptic equations, which coincide with the Euler-Lagrange equation of a convex Lagrangian (see for example \cite[Chapter 1]{Gi}).
\par
This point of view provides an elegant proof, which we believe to be interesting in itself. Moreover,  it enables us to work under minimal assumptions, both on the sets and the functions. In particular, we will make no use of regularity theory for the equation \eqref{LE}, apart for the minimum principle (in measure theoretic sense) for weakly $p-$superharmonic functions. Thus, the proof is very likely to be adapted to more general equations, for example under the presence of weights, where solutions may have a limited degree of regularity. We do not pursue this route in this paper, in order not to bury the main idea under technical details.
\vskip.2cm\noindent
We then apply this comparison principle to obtain a variety of results, notably:
\begin{itemize}
\item a uniqueness result for the minimization of the functional $\mathfrak{F}_{q,\alpha}$ over functions with given (non-negative) boundary datum (Theorem \ref{teo:uniquemin});
\vskip.2cm
\item a ``hierarchy'' result, which asserts that all sign-changing solutions to \eqref{LE} with homogeneous Dirichlet boundary conditions are ``trapped'' between the positive solution $w_{\Omega,\alpha}$ and the negative one $-w_{\Omega,\alpha}$ (Corollary \ref{cor:hierarchy});
\vskip.2cm
\item a sharp {\it pointwise} double-sided estimate on $w_{\Omega,\alpha}$ for convex sets, in terms of geometric quantities (Theorem \ref{thm:bounds}); 
\vskip.2cm
\item  a sharp $L^\infty$ estimate on $w_{\Omega,\alpha}$ (and thus on all sign-changing solutions, see Corollary \ref{cor:Linfty}), as well as a {\it localization result} for maximum points of $w_{\Omega,\alpha}$ (Corollary \ref{cor:localization}), again for convex sets;
\vskip.2cm
\item  a sharp geometric estimate on the so-called {\it generalized principal frequencies} (Theorem \ref{thm:HP}) of an open bounded convex set, i.\,e. the quantities defined by 
\[
\lambda_{p,q}(\Omega) := \inf_{\psi \in C^{\infty}_0(\Omega)} \left\{\int_{\Omega} |\nabla \psi|^p \, dx\, :\, \int_{\Omega} |\psi|^q \, dx=1\right\}.
\]
Observe that the infimum above is attained on $W^{1,p}_0(\Omega)$, under suitable assumption on the open set $\Omega$. By optimality, each minimizer $u$ is a solution of \eqref{LE}, with
$\alpha=\lambda_{p,q}(\Omega)$ and the relevant normalization condition on the $L^q$ norm.
\end{itemize}

\subsection{Some comments} 
At first, we point out that our comparison principle  can be seen as an extension to the quasilinear case of a result by Kajikiya, contained in \cite[Theorem 2.2]{Ka3}. The proof in \cite{Ka3} is different from ours and based on integral identities, relying on Green's formula. A lengthy approximation argument is also needed there, in order to avoid to assume  the regularity of subsolutions and supersolutions (used, on the contrary, in  the comparison principle proved in \cite[Theorem 2.2.4]{PS}). We also observe that the result in \cite{Ka3} is only for the semilinear case, i.e. for $p=2$, but at the same time it is fairly more general, as it deals with positive solutions to
\begin{equation}
\label{eq:general}
-\Delta u = g(u), \quad \mbox{ in }\Omega, 
\end{equation} 
under the assumption that $g$ is sublinear.  
\vskip.2cm\noindent
Our uniqueness result of Theorem \ref{teo:uniquemin} for the minimization of $\mathfrak{F}_{q,\alpha}$ in turn implies uniqueness of the positive solution for the equation \eqref{LE}. In this way, we retrieve a classical result by D\'iaz and Sa\'a contained in \cite{DS}, with a more general boundary datum and without regularity assumptions on the set. On the other hand, we recall that the paper \cite{DS} is concerned with equations fairly more general than \eqref{LE}, where the right-hand side is replaced by a general nonlinearity $g(u)$ which is sub-homogeneous, in a suitable sense.
\par
In any case, we point out that the uniqueness result for the equation does not directly imply the uniqueness of the minimizer, since as already said the functional $\mathfrak{F}_{q,\alpha}$ is not convex in the usual sense. 
Incidentally, we observe that our Theorem \ref{teo:uniquemin} is quite related to \cite[Theorem 1.2]{KLP}, which proves a similar result for $\Omega$ bounded, homogeneous Dirichlet boundary conditions and more general variational integrals with $p-$growth.
\vskip.2cm\noindent
In the subsequent paper \cite{Ka2} Kajikiya applied his comparison principle in order to infer geometric properties of positive solutions to \eqref{eq:general}. Here as well, our paper parallels and extends these results, to the quasilinear case.  In particular our Theorem \ref{thm:bounds} extends \cite[Theorem 2.5]{Ka2} by Kajikiya. Accordingly, all the consequences drawn in \cite{Ka2}, can be inferred here, as well.
\par
We also point out that the double-sided $L^\infty$ estimate of Corollary \ref{cor:Linfty} generalizes to the full range $1\le q<p$ a similar result
for the maximum of the $p-$torsion function, obtained in \cite[Theorem 1.2]{DPGG}. In \cite{DPGG} the result is obtained by means of a different proof, based on the so-called {\it $P-$function method}.
\vskip.2cm\noindent
As for the localization result, our application of the geometric estimate in this context appears to be original. But of course, our result should be compared with related results recently obtained by Magnanini and Poggesi in \cite{MP}, by means of a different proof, still reminiscent of the $P-$function method.
\vskip.2cm\noindent
Finally, concerning the estimate for generalized principal frequencies, our result reads as follows
\begin{equation}
\label{HPBM}
\lambda_{p,q}(\Omega)\,|\Omega|^{\frac{p-q}{q}} \ge \left( \frac{\pi_{p,q}}{2} \right)^p\,\frac{1}{r_\Omega^p},
\end{equation}
for every $\Omega\subset\mathbb{R}^N$ open bounded convex set.
Here $r_\Omega$ is the {\it inradius} of $\Omega$ and $\pi_{p,q}$ is the one-dimensional Sobolev-Poincar\'e constant, defined by
\begin{equation}
\label{pi}
\pi_{p,q} := \inf_{u \in C_0^{\infty}((0,1))} \left\{ \|u'\|_{L^p([0,1])}\, :\, \|u\|_{L^q([0,1])}=1 \right\}. 
\end{equation}
For finite $p$ and $q$, we recall that this constant is computed for example in \cite[equation (7)]{Ta} and it is given by
\[
\pi_{p,q}=\frac{2}{q}\,\left(1+\dfrac{q}{p'}\right)^\frac{1}{q}\,\left(1+\dfrac{p'}{q}\right)^{-\frac{1}{p}}\,B\left(\frac{1}{q},\frac{1}{p'}\right),
\]
where $p'=p/(p-1)$ and $B$ is the {\it Euler Beta function}.
\par
Inequality \eqref{HPBM} is the generalization to $p\not =2$ of a result recently obtained by the first author and Mazzoleni in \cite[Theorem 1.1]{BM}. It contains in a unified way a variety of sharp geometric estimates on generalized principal frequencies. For example, by taking $q=1$ this estimate reduces to 
\[
T_p(\Omega)\le \left(\frac{2}{\pi_{p,1}} \right)^p\,|\Omega|^{p-1}\,r_\Omega^p,
\]
where $T_p(\Omega)=1/\lambda_{p,1}(\Omega)$ is the so-called {\it $p-$torsional rigidity}. This is the extension to the case $p\not=2$ of an old result by Makai \cite{M} in the two-dimensional case. Such an extension
has already been proved in \cite[Theorem 1.1]{DPGG}, by means of a different proof, with respect to the one devised here. We refer to \cite{VBP} and \cite{BBP} for further geometric estimates on the $p-$torsional rigidity.
\par
Similarly, by taking the limit as $q$ goes to $p$ in \eqref{HPBM}, we obtain\footnote{We use here that $q\mapsto \lambda_{p,q}(\Omega)$ is left-continuous at $q=p$, if $\Omega\subset\mathbb{R}^N$ is an open bounded set. Actually, such a property is true under fairly more general assumptions, we refer to Proposition \ref{prop:continuity} below.} 
\[
\lambda_{p}(\Omega)\ge \left( \frac{\pi_{p}}{2} \right)^p\,\frac{1}{r_\Omega^p},
\]
where $\lambda_p$ and $\pi_p$ are defined as above, by simply taking $q=p$.
This is the generalization to the case of the $p-$Laplacian of the {\it Hersch-Protter inequality}, originally proved by Hersch  in \cite[Th\'eor\`eme 8.1]{He} for $p=2=N$ and then generalized to every dimension by Protter in \cite[Theorem 2]{Pr} (see also \cite{BGM}). The case $p\not =2$ has been already obtained in \cite{Bra2, DDG, Ka1}.

\subsection{Plan of the paper}
In Section \ref{sec:2} we present all the definitions and basic results, which will be used throughout the whole paper. In particular, we prove an enhanced version of the hidden convexity principle, which permits to identify the equality cases (see Theorem \ref{thm:shc} below).
\par
We start Section \ref{sec:3} with an existence result for the minimization of $\mathfrak{F}_{q,\alpha}$, among functions with given boundary datum. This is quite standard, but we state it in great generality and pay particular attention to the limit case $q=1$, where the functional is not Gateaux differentiable.
\par
Then, in Section \ref{sec:4} we prove the comparison principle for positive supersolutions and subsolutions of the Lane-Emden equation. We apply it to prove uniqueness of minimizers for the functional $\mathfrak{F}_{q,\alpha}$, when the boundary datum is non-negative and $\Omega$ is connected. This in particular permits to define the function $w_{\Omega,\alpha}$, i.e. the unique positive solution of \eqref{LE} with homogeneous Dirichlet conditions.
\par 
All the applications of the comparison principle to geometric estimates are then given in Section \ref{sec:5}. Finally, the paper is complemented by three appendices: while the first two contain standard materials, in Appendix \ref{app:C} we study the asymptotic behaviour of the function $w_{\Omega,\alpha}$ in the set
\[
\Omega_L=\left(-\frac{L}{2},\frac{L}{2}\right)\times(-1,1),
\]
as the parameter $L$ goes to $+\infty$. This is needed to infer sharpness in Theorem \ref{thm:bounds}.
\begin{ack}
We thank Bruno Volzone for pointing out the reference \cite{SV} and   an  anonymous referee for his suggestions and
comments. The second and third authors are members of the Gruppo Nazionale per l'Analisi Matematica, la Probabilit\`a
e le loro Applicazioni (GNAMPA) of the Istituto Nazionale di Alta Matematica (INdAM). 
\par
Part of this work has been written during the conference ``Variational methods and applications'', held at the {\it Centro di Ricerca Matematica ``Ennio De Giorgi''} in September 2021. The organizers and the hosting institution are gratefully acknowledged. 
\end{ack}

\section{Preliminaries}
\label{sec:2}

\subsection{Notation}

For $x_0\in\mathbb{R}^N$ and $R>0$, we will denote by $B_R(x_0)$ the $N-$dimensional open ball with radius $R$, centered at $x_0$. When the center coincides with the origin, we will simply write $B_R$.
\par
For a function $u\in L^1_{\rm loc}(\mathbb{R}^N)$, we will indicate by $u_+$ and $u_-$ the two functions
\[
u_+=\max\{u,0\}\qquad \mbox{ and }\qquad u_-=\max\{-u,0\}.
\]
In this way, we have $u=u_+-u_-$.
\par
For an open set $\Omega \subset \mathbb{R}^N$ with non-empty boundary, 
we denote by $d_{\Omega}$ the distance function from the boundary $\partial\Omega$, defined by
\[ 
d_{\Omega}(x):=\inf_{y \in \partial\Omega} |x-y|, \qquad \mbox{ for every } x \in \Omega.
\]
The {\it inradius} $r_\Omega$ of $\Omega$ will be the radius of a largest ball contained in $\Omega$. More precisely, this quantity is given by
\begin{equation}
\label{inradius}
r_\Omega=\sup \Big\{r>0\, :\, \mbox{ exists }x_0\in \Omega \mbox{ such that } B_r(x_0)\subset\Omega\Big\}.
\end{equation}
It is well-known that this coincides with the supremum over $\Omega$ of $d_\Omega$.

\subsection{Sobolev spaces} Let $1<p<\infty$ and let $\Omega\subset\mathbb{R}^N$ be an open set. We indicate by $\mathcal{D}^{1,p}_0(\Omega)$ the completion of $C^\infty_0(\Omega)$, with respect to the norm
\[
\psi\mapsto \|\nabla \psi\|_{L^p(\Omega;\mathbb{R}^N)},\qquad \mbox{ for every } \psi\in C^\infty_0(\Omega).
\]
We will also indicate by $W^{1,p}_0(\Omega)$ the closure of $C^\infty_0(\Omega)$, in the usual Sobolev space $W^{1,p}(\Omega)$, endowed with the norm
\[
\|\psi\|_{W^{1,p}(\Omega)}=\|\psi\|_{L^p(\Omega)}+\|\nabla \psi\|_{L^p(\Omega;\mathbb{R}^N)}.
\]
\begin{defn}
\label{defn:qadmissible}
Let $1<p<\infty$, we say that an open set $\Omega\subset\mathbb{R}^N$ is {\it $q-$admissible} for some $1\le q<p$, if 
\[
\lambda_{p,q}(\Omega) := \inf_{\psi \in C^{\infty}_0(\Omega)} \left\{\int_{\Omega} |\nabla \psi|^p \, dx\, :\, \int_{\Omega} |\psi|^q \, dx=1\right\}>0.
\]
This is equivalent to require that the homogeneous Sobolev space $\mathcal{D}^{1,p}_0(\Omega)$ is continuously embedded in $L^q(\Omega)$.
\end{defn}
\begin{rem}
A necessary and sufficient condition for an open set to be $q-$admissible is given in \cite[Theorem 15.5.2]{Maz}. We recall that bounded open sets and, more generally, open sets with finite volume are $q-$admissible, for every $1\le q<p$. However, the class of $q-$admissible sets is larger, since there exist examples of $q-$admissible sets, which have infinite volume (see \cite[Section 15.5.3]{Maz}).  
\end{rem}
\begin{prop}
\label{prop:continuity}
Let $1<p<\infty$ and let $\Omega\subset\mathbb{R}^N$ be an open set, which is $q_0-$admissible for some $1\le q_0<p$. Then we have 
\[
\lim_{q\nearrow p} \lambda_{p,q}(\Omega)=\lambda_{p}(\Omega),
\]
where 
\[
\lambda_{p}(\Omega) := \inf_{\psi \in C^{\infty}_0(\Omega)} \left\{\int_{\Omega} |\nabla \psi|^p \, dx\, :\, \int_{\Omega} |\psi|^p \, dx=1\right\}.
\]
\end{prop}
\begin{proof}
Let $\psi\in C^\infty_0(\Omega)\setminus\{0\}$, by definition of $\lambda_{p,q}(\Omega)$ we have for every $1\le q<p$
\[
\lambda_{p,q}(\Omega)\le \frac{\displaystyle\int_\Omega |\nabla \psi|^p\,dx}{\left(\displaystyle\int_\Omega |\psi|^q\,dx\right)^\frac{p}{q}}.
\]
If we now take the limit as $q$ goes to $p$, we get
\[
\limsup_{q\nearrow p} \lambda_{p,q}(\Omega)\le \lim_{q\nearrow p }\frac{\displaystyle\int_\Omega |\nabla \psi|^p\,dx}{\left(\displaystyle\int_\Omega |\psi|^q\,dx\right)^\frac{p}{q}}=\frac{\displaystyle\int_\Omega |\nabla \psi|^p\,dx}{\displaystyle\int_\Omega |\psi|^p\,dx}.
\]
By arbitrariness of $\psi$ and recalling the definition of $\lambda_p(\Omega)$, we obtain
\[
\limsup_{q\nearrow p} \lambda_{p,q}(\Omega)\le \lambda_p(\Omega).
\]
In order to complete the proof, we use that for every $q_0<q<p$ and every $\psi\in C^\infty_0(\Omega)\setminus\{0\}$ we have
\[
\|\psi\|_{L^q(\Omega)}\le \|\psi\|_{L^{q_0}(\Omega)}^{1-\vartheta}\,\|\psi\|_{L^p(\Omega)}^\vartheta,\qquad \mbox{ with } \vartheta=\frac{p}{q}\,\frac{q-q_0}{p-q_0},
\]
by interpolation in Lebesgue spaces. In particular, by using that $\Omega$ is $q_0-$admissible, this entails that
\[
\|\psi\|_{L^q(\Omega)}\le \left(\frac{1}{\lambda_{p,q_0}(\Omega)}\right)^\frac{1-\vartheta}{p}\,\|\nabla\psi\|_{L^p(\Omega;\mathbb{R}^N)}^{1-\vartheta}\,\|\psi\|_{L^p(\Omega)}^\vartheta.
\]
We can thus estimate the Rayleigh--type quotient defining $\lambda_{p,q}(\Omega)$ as follows
\[
\begin{split}
\frac{\displaystyle\int_\Omega |\nabla \psi|^p\,dx}{\left(\displaystyle\int_\Omega |\psi|^q\,dx\right)^\frac{p}{q}}\ge \Big(\lambda_{p,q_0}(\Omega)\Big)^{1-\vartheta}\,\frac{\left(\displaystyle\int_\Omega |\nabla \psi|^p\,dx\right)^\vartheta}{\left(\displaystyle\int_\Omega |\psi|^p\,dx\right)^\vartheta}\ge \Big(\lambda_{p,q_0}(\Omega)\Big)^{1-\vartheta}\,\Big(\lambda_{p}(\Omega)\Big)^\vartheta.
\end{split}
\]
By arbitrariness of $\psi$, this yields
\[
\lambda_{p,q}(\Omega)\ge \Big(\lambda_{p,q_0}(\Omega)\Big)^{1-\vartheta}\,\Big(\lambda_{p}(\Omega)\Big)^\vartheta.
\]
If we pass to the limit as $q$ goes to $p$ and observe that 
\[
\lim_{q\nearrow p} \vartheta=\lim_{q\nearrow p}\frac{p}{q}\,\frac{q-q_0}{p-q_0}=1,
\]
we finally get
\[
\liminf_{q\nearrow p} \lambda_{p,q}(\Omega)\ge \lambda_p(\Omega),
\]
as desired.
\end{proof}
Let $1\le q<\infty$ and $1<p<\infty$, let $\Omega\subset\mathbb{R}^N$ be an open set. In the sequel, we will need the Sobolev space
\[
X^{q,p}(\Omega):=\Big\{\psi \in L^q(\Omega)\, :\, \nabla \psi \in L^p(\Omega;\mathbb{R}^N)\Big\},
\]
endowed with the norm
\[
\|\psi\|_{X^{q,p}(\Omega)}:=\|\psi\|_{L^q(\Omega)}+\|\nabla \psi\|_{L^p(\Omega;\mathbb{R}^N)}.
\]
Observe that in general we have 
\[
W^{1,p}(\Omega)\not= X^{q,p}(\Omega),
\]
unless $q=p$. For $1\le q<p$, this can be seen for example by adapting the classical counter-example by Nikod\'ym (see \cite[Example 1, Section 1.1.4]{Maz}), to infer existence of an open set $\Omega\subset\mathbb{R}^N$ and a function $\psi_0\in X^{q,p}(\Omega)$ such that
\[
\psi_0\not\in L^p(\Omega).
\]
For $q>p$, it is sufficient to take any irregular open set $\Omega\subset\mathbb{R}^N$ such that $W^{1,p}(\Omega)$ does not imbed into $L^q(\Omega)$ (see for example \cite[page 95]{Gi}).
\par
We also introduce the space $X^{q,p}_0(\Omega)$, defined as the closure of $C^\infty_0(\Omega)$ in $X^{q,p}(\Omega)$. The following technical results will be useful.
\begin{lemma}
\label{lem:Upos}
Let $v,U\in X^{q,p}(\Omega)$, then: 
\begin{enumerate}
\item[{\it (i)}] if $v-U\in X^{q,p}_0(\Omega)$, we have $|v|-|U|\in X^{q,p}_0(\Omega)$, as well;
\vskip.2cm
\item[{\it (ii)}] if $v\in X^{q,p}_0(\Omega)$ and $U$ is non-negative, we have $(v-U)_+\in X^{q,p}_0(\Omega)$.
\end{enumerate}
\end{lemma}
\begin{proof}
We prove $(i)$ first. By assumption, there exists a sequence $\{\psi_n\}_{n\in\mathbb{N}}\subset C^\infty_0(\Omega)$ such that
\[
\lim_{n\to\infty} \|\psi_n-(v-U)\|_{X^{q,p}(\Omega)}=0.
\]
We then define the new sequence $\{\varphi_n\}_{n\in\mathbb{N}}$ by 
\[
\varphi_n=|\psi_n+U|-|U|,\qquad \mbox{ for every } n\in\mathbb{N}.
\]
Observe that each $\varphi_n\in X^{q,p}(\Omega)$ and it has compact support. Thus, with a minor modification of the proof of \cite[Lemma 9.5]{Brezis}, we get $\{\varphi_n\}_{n\in\mathbb{N}}\subset X^{q,p}_0(\Omega)$. By construction and by the triangle inequality, we have
\[
\Big|\varphi_n-(|v|-|U|)\Big|=\Big||\psi_n+U|-|v|\Big|\le |\psi_n-(v-U)|.
\]
Thus we get
\[
\lim_{n\to\infty}\big\|\varphi_n-(|v|-|U|)\big\|_{L^q(\Omega)}=0.
\]
Moreover, it is easily seen that the sequence $\{\nabla \varphi_n\}_{n\in\mathbb{N}}\subset L^p(\Omega;\mathbb{R}^N)$ is bounded. Thus 
it is weakly converging to some $\phi\in L^p(\Omega;\mathbb{R}^N)$, up to a subsequence. By using the strong convergence in $L^q(\Omega)$ and the definition of weak gradient, we can identify $\phi=\nabla |v|-\nabla |U|$. By Mazur's Lemma (see \cite[Theorem 2.13]{LL}), we can build a new sequence $\{(\widetilde\varphi_n,\widetilde\phi_n)\}_{n\in\mathbb{N}}\subset L^q(\Omega)\times L^p(\Omega;\mathbb{R}^N)$, made of convex combinations of $\{(\varphi_n,\nabla \varphi_n)\}_{n\in\mathbb{N}}$, such that 
\[
\lim_{n\to\infty} \left( \|\widetilde\varphi_n-(|v|-|U|)\|_{L^q(\Omega)}+\|\widetilde\phi_n-\nabla(|v|-|U|)\|_{L^p(\Omega;\mathbb{R}^N)}\right)=0.
\]
Moreover, by construction we clearly have $\widetilde\phi_n=\nabla \widetilde\varphi_n$. This permits to show that $|v|-|U|$ is the limit in the $X^{q,p}(\Omega)$ norm of a sequence $\{\widetilde\varphi_n\}_{n\in\mathbb{N}}\subset X^{q,p}_0(\Omega)$. Since the latter is a closed subspace, we get the conclusion.
\vskip.2cm\noindent
In order to prove $(ii)$, it is sufficient to write
\[
(v-U)_+=\frac{|v-U|+(v-U)}{2}=\frac{(|v-U|-U)+v}{2}.
\]
Then we observe that $v\in X^{q,p}_0(\Omega)$ by assumption, while 
\[
|v-U|-U=|U-v|-|U|\in X^{q,p}_0(\Omega),
\]
from point $(i)$, applied to the function $U-v$. Indeed, observe that $(U-v)-U\in X^{q,p}_0(\Omega)$.
\end{proof}

\begin{prop}
\label{prop:spazi0}
Let $1\le q<p<\infty$ and let $\Omega\subset\mathbb{R}^N$ be a $q-$admissible open set. Then we have
\[
X^{q,p}_0(\Omega)=W^{1,p}_0(\Omega)=\mathcal{D}^{1,p}_0(\Omega),
\] 
and the three spaces are compactly embedded into $L^q(\Omega)$.
Consequently, we also have 
\[
W^{1,p}_0(\Omega)\cap X^{q,p}(\Omega)=W^{1,p}_0(\Omega).
\]
\end{prop}
\begin{proof}
It is sufficient to prove the first fact, the second one being an easy consequence of this.
\par
We need to prove that the three norms
\[
\|\nabla \psi\|_{L^p(\Omega;\mathbb{R}^N)},\qquad \|\psi\|_{W^{1,p}(\Omega)} \qquad \mbox{ and }\qquad \|\psi\|_{X^{q,p}(\Omega)},
\]
are equivalent on $C^\infty_0(\Omega)$.  
Observe that under the assumption of $q-$admissibility of the set $\Omega$, we have
\[
\|\psi\|_{X^{q,p}(\Omega)}\le \left(\lambda_{p,q}(\Omega)^{-\frac{1}{p}}+1\right)\,\|\nabla \psi\|_{L^p(\Omega;\mathbb{R}^N)},\qquad \mbox{ for every } \psi\in C^\infty_0(\Omega).
\]
Moreover, we recall the {\it Gagliardo-Nirenberg interpolation inequality}
\begin{equation}
\label{GNS}
\|\psi\|_{L^p(\Omega)}\le C\,\|\psi\|_{L^q(\Omega)}^{1-\vartheta}\,\|\nabla \psi\|_{L^p(\Omega;\mathbb{R}^N)}^\vartheta,\qquad \mbox{ for every } \psi\in C^\infty_0(\Omega),
\end{equation}
where the exponent $\vartheta=\vartheta(N,p,q)\in (0,1)$ is dictated by scale invariance, i.e.
\[
\vartheta=\frac{\dfrac{N}{q}-\dfrac{N}{p}}{\dfrac{N}{q}-\dfrac{N}{p}+1},
\]
and $C=C(N,p,q)>0$. By using also Young's inequality with conjugate exponents $1/\vartheta$ and $1/(1-\vartheta)$, we get from \eqref{GNS}
\[
\begin{split}
\|\psi\|_{W^{1,p}(\Omega)}&=\|\psi\|_{L^p(\Omega)}+\|\nabla\psi\|_{L^p(\Omega;\mathbb{R}^N)}\\
&\le C'\left(\,\|\psi\|_{L^q(\Omega)}+\,\|\nabla\psi\|_{L^p(\Omega;\mathbb{R}^N)}\right)+\|\nabla \psi\|_{L^p(\Omega;\mathbb{R}^N)}\\
&\le (C'+1)\,\|\psi\|_{X^{q,p}(\Omega)},\qquad\qquad \mbox{ for every } \psi\in C^\infty_0(\Omega).
\end{split}
\]
The previous estimates show the desired equivalence of the norms on $C^\infty_0(\Omega)$. As for the compactness of the embedding, it is sufficient to recall that for $1\le q<p$ the embedding of $\mathcal{D}^{1,p}_0(\Omega)$ into $L^q(\Omega)$ is continuous if and only if it is compact (see \cite[Theorem 15.6.2]{Maz} or also \cite[Theorem 1.2]{BR}). This fact and the equivalence of the spaces conclude the proof.
\end{proof}
%\begin{rem}
%\label{rem:banale}
%Under the previous assumptions, by definition of the space $X^{q,p}_0(\Omega)$, it is clear that inequality \eqref{interpolation} still holds for $\psi\in X^{q,p}_0(\Omega)$.
%\end{rem}

\subsection{Weak solutions}
We recall the following definition
\begin{defn}
Let $1<p<\infty$ and let $\Omega\subset\mathbb{R}^N$ be an open set. We say that $u\in L^1_{\rm loc}(\Omega)$ is {\it weakly $p-$superharmonic in $\Omega$} if $\nabla u \in L^p(\Omega;\mathbb{R}^N)$ and  
\begin{equation}\label{eq:up-sup}
\int_{\Omega} \langle |\nabla u|^{p-2} \,\nabla u, \nabla \psi \rangle \, dx \ge 0, \qquad \mbox{ for every } \psi \in C^{\infty}_0(\Omega), \, \psi \ge 0.
\end{equation}
\end{defn}
\begin{defn}
Let $\alpha>0$  and $1<  q < p<\infty$. Let $\Omega\subset\mathbb{R}^N$ be an open set. We say that a function $v \in X^{q,p}(\Omega)$ is a: 
\begin{itemize}
\item {\it weak supersolution of} \eqref{LE} if 
\[
\int_{\Omega} \langle |\nabla v|^{p-2} \nabla v, \nabla \psi \rangle \, dx \ge \alpha\int_{\Omega} |v|^{q-2}\,v\, \psi \, dx, \quad \text{ for every } \psi \in C_0^{\infty}(\Omega),\ \psi \ge 0;
\]
\item {\it weak subsolution of} \eqref{LE} if
\[
\int_{\Omega} \langle |\nabla v|^{p-2} \, \nabla v, \nabla \psi \rangle \, dx \le \alpha\, \int_{\Omega} |v|^{q-2}\,v\, \psi \, dx, \quad \text{ for every } \psi \in C_0^{\infty}(\Omega),\ \psi \ge 0;
\]
\item {\it weak solution of} \eqref{LE} if 
\[
\int_{\Omega} \langle |\nabla v|^{p-2} \,\nabla v, \nabla \psi \rangle \, dx = \alpha\, \int_\Omega |v|^{q-2}\, v\, \psi \, d x, \quad \text{ for every } \psi \in C_0^{\infty}(\Omega). 
\]
\end{itemize}
We can obviously admit test functions in $X^{q,p}_0(\Omega)$ in the above inequalities, by a standard density argument.
\par
In the case $q=1$, for a non-negative function $v$ we extend the previous definitions, by 
using the convention 
\[
|v|^{q-2}\,v=v^{q-1}=1.
\] 
\end{defn}
\begin{rem}[Scalings]
\label{rem:scalings}
It is easily seen that if $u \in X^{q,p}(\Omega)$ is a weak solution of 
\[
-\Delta_p u=\alpha\,|u|^{q-2}\,u,\qquad \mbox{ in }\Omega,
\]
and $t>0$, then the rescaled function
\[
u_t(x)=t^{\frac{p}{p-q}}\, u \left( \frac{x-x_0}{t} \right),
\]
is a weak solution of the same equation, in the new set $x_0+t\, \Omega$. On the other hand, if $\beta>0$ and we set 
\[
v=\left(\frac{\beta}{\alpha}\right)^\frac{1}{p-q}\,u,
\]
then this is a weak solution of 
\[
-\Delta_p v=\beta\,|v|^{q-2}\,v,\qquad \mbox{ in } \Omega.
\]
The same remarks apply to subsolutions and supersolutions.
\end{rem}

\subsection{Hidden convexity}

We will need the following {\it hidden convexity property} of the $p-$Dirichlet integral. 
This is nowadays well-known among experts, we give here a general statement under minimal assumptions both on the open set and the functions. We also identify the equality cases, by means of an enhanced version of the inequality, see equation \eqref{rinforzino} below.

\begin{thm}[Hidden convexity]\label{thm:shc}
Let $1\le q < \infty$, $1<p<\infty$ and $1\le r \le \min\{p,q\}$. Let $\Omega \subset \mathbb{R}^N$ be an open set, for every pair of non-negative functions $v, w\in X^{q,p}(\Omega)$, we set
\begin{equation}
\label{eq:sigmat}
\sigma^t:=\left((1-t)\, v^r + t\,w^r\right)^\frac{1}{r},\qquad \text{ for every } t \in [0,1].
\end{equation} 
Then $\sigma^t \in X^{q,p}(\Omega)$ and
\begin{equation}\label{eq:hc}
\int_{\Omega} |\nabla \sigma^t|^p \, dx \le (1-t) \int_\Omega |\nabla v|^p \, dx + t \int_\Omega |\nabla w|^p \, dx \quad \text{ for every } t \in [0,1].
\end{equation}
Moreover, when $\Omega$ is connected: 
\begin{itemize}
\item if $\boxed{1=r<p}$, the equality for some $t\in(0,1)$ holds in \eqref{eq:hc} if and only if $v=w+C$, with $C \in \mathbb{R}$;
\vskip.2cm
\item if $\boxed{1<r<p}$, the equality for some $t\in(0,1)$ holds in \eqref{eq:hc} if and only if 
\[
\mbox{ either }\qquad v=w \qquad \mbox{ or }\qquad \mbox{ $w$ and $v$ are both constant};
\]
\item if $\boxed{1<r=p}$ and in addition we have
\begin{equation}
\label{ipotesa}
\frac{1}{v} \in L^{\infty}_{\rm loc}(\Omega)\qquad \mbox{ and }\qquad \frac{1}{w}\in L^\infty_{\rm loc}(\Omega),
\end{equation}
the equality for some $t\in(0,1)$ holds in \eqref{eq:hc} if and only if $v=C\, w$, with $C>0$.
\end{itemize}
\end{thm}
\begin{proof}	
For $r=1$ there is nothing to prove, in this case $\sigma^t$ is just the usual convex combination of $v$ and $w$. Then \eqref{eq:hc} is just the usual convexity of the $p-$Dirichlet integral. As for the equality cases, from the strict convexity of the map $z\mapsto |z|^p$, we get that if \eqref{eq:hc} holds as an identity for some $t\in(0,1)$, then we must have
\[
\nabla v=\nabla w\qquad \mbox{ a.\,e. in }\Omega.
\]
If $\Omega$ is connected, this implies that $v-w$ must be constant in $\Omega$.
\vskip.2cm\noindent
We now focus on the case $1<r\le \min\{p,\,q\}$. We will first show that $\sigma^t$ belongs to $X^{q,p}(\Omega)$ and satisfies \eqref{eq:hc}. Then we will focus on the equality cases in the latter, under the claimed assumptions. 
\vskip.2cm\noindent
{\bf Part 1: properties of $\sigma^t$}. It is easy to prove that $\sigma^t \in L^q(\Omega)$. Indeed, observe that
\[
\int_\Omega |\sigma^t|^q\,dx=\int_\Omega \left((1-t)\, v^r + t\,w^r\right)^\frac{q}{r}\,dx\le (1-t)\,\int_\Omega v^q\,dx+t\,\int_\Omega w^q\,dx<+\infty,
\] 
thanks to the convexity of the map $\tau\mapsto \tau^{q/r}$.
For $\varepsilon > 0$ and $t\in (0,1)$, we consider the $C^1$ function $G_{\varepsilon,t}$ defined on $[0,+\infty)\times[0,+\infty)$ by 
\[ 
G_{\varepsilon,t}(s_1,s_2):=\left((1-t)\,(s_1+\varepsilon)^r+t\,(s_2+\varepsilon)^r\right)^\frac{1}{r}-\varepsilon.
\]
Then, we take the vector field
\[ 
\Phi(x)= (v(x),w(x)) \in X^{q,p}(\Omega; \mathbb{R}^2). 
\]
Since  
\[ 
\nabla G_{\varepsilon,t}(s_1,s_2)=\left(\frac{(1-t)\,(s_1+\varepsilon)^{r-1}}{\left((1-t)\,(s_1+\varepsilon)^r+t\,(s_2+\varepsilon)^r\right)^\frac{r-1}{r}},\,\frac{t\,(s_2+\varepsilon)^{r-1}}{\left((1-t)\,(s_1+\varepsilon)^r+t\,(s_2+\varepsilon)^r\right)^\frac{r-1}{r}}\right),
\] 
it is easily seen that $G_{\varepsilon,t}$ has a bounded gradient.
Thanks to the Chain Rule in Sobolev spaces (see for example \cite[Theorem 6.16]{LL}), we have that
\[
\sigma_\varepsilon^t=G_{\varepsilon,t} \circ \Phi\in X^{q,p}(\Omega).
\]
Observe that we also used that $G_{\varepsilon,t}(0,0)=0$, to guarantee that $\sigma_\varepsilon^t$ has the required summability, when $\Omega$ has infinite volume.
\par
We remark that for every $s_1,s_2\ge 0$ and $t\in(0,1)$, the quantity $\varepsilon\mapsto G_{\varepsilon,t}(s_1,s_2)$ is monotone decreasing. Indeed, observe that 
\[
\frac{d}{d\varepsilon} G_{\varepsilon,t}(s_1,s_2)=\frac{(1-t)\,(s_1+\varepsilon)^{r-1}+t\,(s_2+\varepsilon)^{r-1}}{\left((1-t)\,(s_1+\varepsilon)^r+t\,(s_2+\varepsilon)^r\right)^\frac{r-1}{r}}-1\le 0,
\]
thanks to the concavity of the map $\tau\mapsto \tau^{(r-1)/r}$. 
Thus in particular we have
\[
\sigma^t_{\varepsilon_2}\le \sigma^t_{\varepsilon_1}\le \sigma^t,\qquad \mbox{ for } 0<\varepsilon_1<\varepsilon_2.
\]
By using the Monotone Convergence Theorem and the fact that $\sigma^t\in L^q(\Omega)$, we thus obtain
\[
\lim_{\varepsilon\to 0} \int_\Omega |\sigma^t_\varepsilon-\sigma^t|^q\,dx=\lim_{\varepsilon\to 0} \int_\Omega (\sigma^t-\sigma^t_\varepsilon)^q\,dx=0.
\]
We now compute the gradient of $\sigma_\varepsilon^t$ 
\begin{equation}
\label{gradientone}
\begin{split}
\nabla \sigma_\varepsilon^t&=\frac{(1-t)\,(v+\varepsilon)^{r-1}}{\left((1-t)\,(v+\varepsilon)^r+t\,(w+\varepsilon)^r\right)^\frac{r-1}{r}}\, \nabla v +\frac{t\,(w+\varepsilon)^{r-1}}{\left((1-t)\,(v+\varepsilon)^r+t\,(w+\varepsilon)^r\right)^\frac{r-1}{r}} \, \nabla w\\
&=(G_{\varepsilon,t}(v,w)+\varepsilon)\, \left(
\frac{(1-t)\,(v+\varepsilon)^r}{\left(G_{\varepsilon,t}(v,w)+\varepsilon\right)^r}\, \frac{\nabla v}{v+\varepsilon}
+\frac{t\,(w+\varepsilon)^r}{\left(G_{\varepsilon,t}(v,w)+\varepsilon\right)^r}\,\frac{\nabla w}{w+\varepsilon}\right).
\end{split}
\end{equation}
We now observe that 
\[
\frac{(1-t)\,(v+\varepsilon)^r}{\left(G_{\varepsilon,t}(v,w)+\varepsilon\right)^r}\qquad \mbox{ and }\qquad\frac{t\,(w+\varepsilon)^r}{\left(G_{\varepsilon,t}(v,w)+\varepsilon\right)^r},
\]
are both non-negative, less than or equal to $1$ and their sum gives $1$, thanks to the definition of $G_{\varepsilon,t}$. Thus they can be regarded as the coefficients of a convex combination. By using the convexity of the map $z\mapsto |z|^r$, we thus obtain 
\[
|\nabla \sigma_\varepsilon^t|^r \le(G_{\varepsilon,t}(v,w)+\varepsilon)^r \,\left(\frac{(1-t)\,(v+\varepsilon)^r}{\left(G_{\varepsilon,t}(v,w)+\varepsilon\right)^r}\, \left|\frac{\nabla v}{v+\varepsilon}\right|^r
+\,\frac{t\,(w+\varepsilon)^r}{\left(G_{\varepsilon,t}(v,w)+\varepsilon\right)^r}\,\left|\frac{\nabla w}{w+\varepsilon}\right|^r\right).
\]
After some simplifications the previous estimate can be rewritten as
\begin{equation}
\label{semplice}
|\nabla \sigma_\varepsilon^t|^r\le (1-t)\,|\nabla v|^r+t\,|\nabla w|^r.
\end{equation}
We now raise to the power $p/r$, integrate over $\Omega$ and use the convexity of the map $\tau\mapsto \tau^{p/r}$. We finally obtain
\begin{equation}
\label{quasiconvex}
\int_\Omega |\nabla \sigma_\varepsilon^t|^p\,dx\le (1-t)\,\int_\Omega |\nabla v|^p\,dx+t\,\int_\Omega |\nabla w|^p\,dx.
\end{equation}
This shows that $\{\nabla \sigma_\varepsilon^t\}_{\varepsilon>0}$ is a bounded subset of $L^p(\Omega;\mathbb{R}^N)$, for every $t \in (0,1)$. Thus, there exists an infinitesimal sequence $\{\varepsilon_n\}_{n\in\mathbb{N}}\subset (0,1)$ such that $\nabla \sigma_{\varepsilon_n}^t$ weakly converges to a vector field $\phi^t\in L^p(\Omega;\mathbb{R}^N)$. By using the definition of weak gradient and the strong convergence in $L^q$ previously inferred, it is standard to show that we must have $\phi^t=\nabla \sigma^t$.
This proves that
\[
\sigma^t\in X^{q,p}(\Omega),
\] 
as claimed. Moreover, by taking the limit in \eqref{quasiconvex} and using the lower semicontinuity of the $L^p$ norm with respect to the weak convergence, we finally establish \eqref{eq:hc}, at once.
\vskip.2cm\noindent
{\bf Part 2: enhanced hidden convexity.} We go back for a moment to \eqref{gradientone}. Then, rather than simply using the convexity of $z\mapsto|z|^r$, we use a ``quantified'' version of such a property. This is expressed by the inequalities of Lemma \ref{lm:lambdaconvexgen} and Lemma \ref{lm:lambdaconvexgensub}. Thus we now obtain:
\begin{itemize}
\item if $r\ge 2$, in place of \eqref{semplice} we get
\[
|\nabla \sigma^t_\varepsilon|^r+C\, \frac{t\,(1-t)\,(w+\varepsilon)^r\,(v+\varepsilon)^r}{\left(G_{\varepsilon,t}(v,w)+\varepsilon\right)^r}\,\left|\frac{\nabla v}{v+\varepsilon}-\frac{\nabla w}{w+\varepsilon}\right|^r\le (1-t)\,|\nabla v|^r+t\,|\nabla w|^r,
\]
which can be further simplified into
\[
|\nabla \sigma^t_\varepsilon|^r +C\,t\,(1-t)\,\frac{\left|(w+\varepsilon)\,\nabla v-(v+\varepsilon)\,\nabla w\right|^r}{\left(G_{\varepsilon,t}(v,w)+\varepsilon\right)^r}\,\le (1-t)\,|\nabla v|^r+t\,|\nabla w|^r;
\]
\item if $1<r<2$, in place of \eqref{semplice} we get
\[
\begin{split}
|\nabla \sigma^t_\varepsilon|^r&+C\,\frac{t\,(1-t)\,(w+\varepsilon)^r\,(v+\varepsilon)^r}{\left(G_{\varepsilon,t}(v,w)+\varepsilon\right)^r}\,\left(\left|\frac{\nabla v}{v+\varepsilon}\right|^2+\left|\frac{\nabla w}{w+\varepsilon}\right|^2\right)^\frac{r-2}{2}\,\left|\frac{\nabla v}{v+\varepsilon}-\frac{\nabla w}{w+\varepsilon}\right|^2\\
&\le (1-t)\,|\nabla v|^r+t\,|\nabla w|^r,
\end{split}
\]
which can be further simplified into
\[
\begin{split}
|\nabla \sigma^t_\varepsilon|^r&+C\,t\,(1-t)\,\frac{\left(\left|(w+\varepsilon)\,\nabla v\right|^2+\left|(v+\varepsilon)\,\nabla w\right|^2\right)^\frac{r-2}{2}}{{\left(G_{\varepsilon,t}(v,w)+\varepsilon\right)^r}}\,\left|(w+\varepsilon)\,\nabla v-(v+\varepsilon)\,\nabla w\right|^2\\
&\le (1-t)\,|\nabla v|^r+t\,|\nabla w|^r.
\end{split}
\]
\end{itemize}
By recalling the definition of $G_{\varepsilon,t}$, it is not difficult to see that for almost every $x\in\Omega$, we have\footnote{We use that for a function $v\in W^{1,1}_{\rm loc}(\Omega)$ we have 
\[
\nabla v=0,\qquad \mbox{ a.\,e. in } \{x\in\Omega\, :\, v(x)=0\},
\]
see for example \cite[Theorem 6.19]{LL}.}
\begin{itemize}
\item if $r\ge 2$
\[
\lim_{\varepsilon\to 0^+} \frac{|(w+\varepsilon)\,\nabla v-(v+\varepsilon)\,\nabla w|^r}{\left(G_{\varepsilon,t}(v,w)+\varepsilon\right)^r}=\mathcal{R}_r(v,w):=\left\{\begin{array}{cc}
\dfrac{|w\,\nabla v-v\,\nabla w|^r}{(1-t)\,v^r+t\,w^r},& \mbox{ if } w(x)+v(x)>0,\\
&\\
0,& \mbox{ if } w(x)=v(x)=0;
\end{array}
\right.
\]
\item while for $1<r<2$
\[
\begin{split}
\lim_{\varepsilon\to 0^+}& \frac{\left(\left|(w+\varepsilon)\,\nabla v\right|^2+\left|(v+\varepsilon)\,\nabla w\right|^2\right)^\frac{r-2}{2}\,\left|(w+\varepsilon)\,\nabla v-(v+\varepsilon)\,\nabla w\right|^2}{\left(G_{\varepsilon,t}(v,w)+\varepsilon\right)^r}\\
&=\mathcal{R}_r(v,w):=\left\{\begin{array}{cc}
\dfrac{\left(\left|w\,\nabla v\right|^2+\left|v\,\nabla w\right|^2\right)^\frac{r-2}{2}\,\left|w\,\nabla v-v\,\nabla w\right|^2}{(1-t)\,v^r+t\,w^r},& \mbox{ if } w(x)+v(x)>0,\\
&\\
0,& \mbox{ if } w(x)=v(x)=0.
\end{array}
\right.
\end{split}
\]
\end{itemize}
As before, we now raise to the power $p/r$ the pointwise estimates above, integrate over $\Omega$ and use the convexity and super-additivity of the map $\tau\mapsto \tau^{p/r}$. A further limit as $\varepsilon$ goes to $0$, in conjuction with Fatou's Lemma, finally gives
\begin{equation}
\label{rinforzino}
\begin{split}
\int_\Omega |\nabla \sigma^t|^p\,dx&+C\, \left(t\,(1-t)\right)^\frac{p}{r}\,\int_\Omega \left(\mathcal{R}_r(v,w)\right)^\frac{p}{r}\,dx\\
&\le \int_\Omega \Big((1-t)\,|\nabla v|^r+t\,|\nabla w|^r\Big)^\frac{p}{r}\,dx\\
&\le (1-t)\,\int_\Omega|\nabla v|^p\,dx+t\,\int_\Omega|\nabla w|^p\,dx.
\end{split}
\end{equation}
{\bf Part 3: equality cases for $r<p$.} We now suppose that $\Omega$ is connected. If equality holds in \eqref{eq:hc} for some $t\in(0,1)$, in particular we must have equality in \eqref{rinforzino}, as well. Thus we have that 
\[
\int_\Omega \Big((1-t)\,|\nabla v|^r+t\,|\nabla w|^r\Big)^\frac{p}{r}\,dx=(1-t)\,\int_\Omega|\nabla v|^p\,dx+t\,\int_\Omega|\nabla w|^p\,dx,
\]
and 
\[
\int_\Omega \left(\mathcal{R}_r(v,w)\right)^\frac{p}{r}\,dx=0.
\]
The first fact implies that 
\begin{equation}
\label{rovazzi}
|\nabla v|=|\nabla w|,\qquad \mbox{ a.\,e. in } \Omega,
\end{equation}
thanks to the strict convexity of $\tau\mapsto \tau^{p/r}$. 
The second fact implies that 
\begin{equation}
\label{storta}
w\,\nabla v=v\,\nabla w,\qquad \mbox{ a.\,e. in }\Omega,
\end{equation}
thanks to the definition of $\mathcal{R}_r(v,w)$. We claim that \eqref{rovazzi} and \eqref{storta} imply that $\nabla w=\nabla v$ almost everywhere in $\Omega$. Indeed, let us call 
\[
E=\Big\{x\in\Omega\, :\, \nabla w(x)\not=\nabla v(x)\Big\},
\]
and let us suppose that $|E|>0$. 
By recalling that $w$ and $v$ are non-negative, from \eqref{storta} we get in particular
\[
w\,|\nabla v|=v\,|\nabla w|,\qquad \mbox{ a.\,e. in } E.
\]
By recalling \eqref{rovazzi} and the definition of $E$, the last identity in turn implies that 
\[
w=v,\qquad \mbox{ a.\,e. in } E.
\]
On the other hand, we know that (see again \cite[Theorem 6.19]{LL})
\[
\nabla w=\nabla v,\qquad \mbox{ a.\,e. in } \Big\{x\in\Omega\, :\, w(x)=v(x)\Big\}.
\] 
The last properties in particular give that
\[
\nabla w=\nabla v,\qquad \mbox{ a.\,e. in } E.
\]
By recalling the definition of $E$, this is a contradiction. This gives that 
\[
\nabla (w-v)=0,\qquad \mbox{ a.\,e. in }\Omega,
\]
and thus $w-v=C$ in $\Omega$ for some constant $C$, since $\Omega$ is connected. We can spend this information in \eqref{storta}, so to get 
\[
(v+C)\,\nabla v=v\,\nabla v,\qquad \mbox{ a.\,e. in }\Omega,
\]
that is $C\,\nabla v$ vanishes almost everywhere in $\Omega$. This implies that either $C=0$ (and thus $w=v$) or $v$ is constant in $\Omega$ (and the same is true for $w$). 
\vskip.2cm\noindent
{\bf Part 4: equality cases for $r=p$.} We suppose again that $\Omega$ is connected and assume the stronger condition \eqref{ipotesa} on $v$ and $w$. If equality holds in \eqref{eq:hc}, from \eqref{rinforzino} we get in particular that 
\[
\int_\Omega \mathcal{R}_p(v,w)\,dx=0.
\]
This gives again \eqref{storta}, as above. The assumption on $v$ and $w$ entails that for every open set $\Omega'$ compactly contained in $\Omega$, there exists $c_{\Omega'}>0$ such that $v,w\ge c_{\Omega'}$ almost everywhere in $\Omega'$. This permits to infer that 
\[
\log v\in W^{1,1}_{\rm loc}(\Omega) \qquad \mbox{ and }\qquad \log w\in W^{1,1}_{\rm loc}(\Omega), 
\]
and the Chain Rule formula holds for their distributional gradients. We then obtain
\[
\nabla (\log v-\log w)=\frac{\nabla v}{v}-\frac{\nabla w}{w}=0,\qquad \mbox{ a.\,e. in }\Omega,
\]
thanks to \eqref{storta}. Since $\Omega$ is an open connected set, the difference $\log v-\log w$ must be constant almost everywhere in $\Omega$. This in turn permits to conclude that $v=C\,w$ almost everywhere in $\Omega$, for some $C>0$. The proof is now concluded.
\end{proof}
\begin{rem}
\label{rem:HC}
As already said in the Introduction, the previous property of the $p-$Dirichlet integral has been discovered by Benguria, in the case $p=r=2$ (see \cite{Be1,BBL}) and then generalized to $p=r\not= 2$ by D\'iaz and Sa\'a in \cite{DS}. Since then, it has been extended by various authors, to cover the case of a general $1<p<\infty$ and an exponent $1<r\le p$. We refer for example to \cite{BK, Ka, KLP, Na}, as well as to \cite{BFMST, BF1, Diaz, TTU}, where the same property for more general Dirichlet--type integrals is proved. %We also point out the paper \cite{KLP}, which contains the identification of equality cases, as well. 
%The main difference between the proof \cite{KLP} and ours 
The main novelty with respect to the above references is that we use a qualified form of convexity for the power $z\mapsto |z|^r$, which gives as a bonus the enhanced inequality \eqref{rinforzino}. This in turn implies that we can identify the equality cases. 
Finally, we recall the reference \cite{BF2}, where the connections of the hidden convexity property with the so-called {\it Picone inequality} are explained.
\end{rem}
Actually, if the curve $t\mapsto\sigma^t$ of the previous result is built from two Sobolev functions sharing the same boundary datum, the same remains true for the curve itself. This is the content of the next result.
\begin{prop}
\label{prop:datumU}
Let $1\le q<\infty$, $1 < p< \infty$ and $1\le r \le \min\{p,q\}$. Let $\Omega \subset \mathbb{R}^N$ be an open set, for every pair of non-negative functions $v, w\in X^{q,p}(\Omega)$, we still denote by $\sigma^t$ the curve of functions defined by \eqref{eq:sigmat}.
If there exists $U\in X^{q,p}(\Omega)$ such that 
\[
v-U\in X^{q,p}_0(\Omega)\qquad \mbox{ and }\qquad w-U\in X^{q,p}_0(\Omega),
\]
then
\[
\sigma^t-U \in X_0^{q,p}(\Omega).
\]
\end{prop}

\begin{proof}
By assumption, there exist two sequences $\{\psi_n\}_{n\in\mathbb{N}}, \{\varphi_n\}_{n\in\mathbb{N}} \subset C^{\infty}_0(\Omega)$ such that 
\[
\lim_{n\to\infty}\|\psi_n -(v - U)\|_{X^{q,p}(\Omega)}=0,
\]
and
\[
\lim_{n\to\infty}\|\varphi_n -(w - U)\|_{X^{q,p}(\Omega)}=0 . 
\]
Then we set 
\[
v_n:=\psi_n + U\in X^{q,p}(\Omega)\qquad \mbox{ and }\qquad w_n:=\varphi_n + U \in X^{q,p}(\Omega),
\] 
and observe that the first one converges to $v$, while the second one to $w$.
For every $t \in [0,1]$, thanks to Theorem \ref{thm:shc}, we know that
\[ 
\sigma_{n}^t:=((1-t)\,v_n^r+t\, w_n^r)^\frac{1}{r} \in X^{q,p}(\Omega). 
\]
Moreover, since by construction $\sigma_{n}^t-U$ has compact support in $\Omega$, we have that $\sigma_{n}^t - U \in X^{q,p}_0(\Omega)$ (it is sufficient to adapt the proof of \cite[Lemma 9.5]{Brezis} to our Sobolev space). In order to conclude the proof, it is sufficient to show that
\begin{equation}
\label{funzioni}
\lim_{n\to\infty} \|\sigma_n^t-\sigma^t\|_{L^q(\Omega)}=0,\qquad \mbox{ for every } t\in (0,1),
\end{equation} 
and
\begin{equation}
\label{gradienti}
\sup_{n\in\mathbb{N}}\|\nabla \sigma_n^t\|_{L^p(\Omega;\mathbb{R}^N)}<+\infty,\qquad \mbox{ for every } t\in(0,1).
\end{equation} 
Indeed, thanks to the reflexivity of $L^p(\Omega;\mathbb{R}^N)$, from \eqref{gradienti} we would get that $\nabla \sigma_n^t-\nabla U$ weakly converges to some $\phi^t\in L^p(\Omega;\mathbb{R}^N)$, up to a subsequence. Then \eqref{funzioni} and the definition of weak gradient would permit to show that $\phi^t=\nabla \sigma^t-\nabla U$. Thus $\sigma^t-U$ would coincide with the weak limit of a sequence in $X^{q,p}_0(\Omega)$. As in the proof of Lemma \ref{lem:Upos}, we can appeal to Mazur's Lemma to show that $\sigma^t-U$ is also a strong limit of a sequence in $X^{q,p}_0(\Omega)$. We finally conclude that $\sigma^t-U \in X^{q,p}_0(\Omega)$, as claimed.
\par
The strong convergence in $L^q$ can be obtained by observing that\footnote{We denote by $\|z\|_{\ell^r}$ the norm
\[
\|z\|_{\ell^r}=\left(|z_1|^r+|z_2|^r\right)^\frac{1}{r},\qquad \mbox{ for every } z=(z_1,z_2)\in\mathbb{R}^2.
\]}
\[ 
\sigma^t = \left\| \left( (1-t)^{1/r}\, v, t^{1/r}\, w \right) \right\|_{\ell^r}\quad \mbox{ and }\quad \sigma_{n}^t = \left\| \left( (1-t)^{1/r}\, v_n, t^{1/r}\, w_n \right) \right\|_{\ell^r}.
\]
Then, by using the triangle inequality for the $\ell^r$ norm
\[
\Big| \|z\|_{\ell^r} - \|\xi\|_{\ell^r} \Big| \le \| z -\xi \|_{\ell^r} \quad \text{ for } z,\xi \in \mathbb{R}^2,
\]
it follows
\[
|\sigma_{n}^t - \sigma^t| \le \left((1-t)\, |v_n-v|^r + t\, |w_n-w|^r\right)^\frac{1}{r}.
\]
By raising to the power $q$, integrating over $\Omega$ and using the convexity of the map $\tau\mapsto\tau^{q/r}$, we get
\[
\int_\Omega |\sigma_{n}^t - \sigma^t|^q\,dx\le (1-t)\,\int_\Omega |v_n-v|^q\,dx+t\,\int_\Omega |w_n-w|^q\,dx.
\] 
By recalling that $\{v_n\}_{n\in\mathbb{N}}$ and $\{w_n\}_{n\in\mathbb{N}}$ strongly converge in $L^{q}(\Omega)$ to $v$ and $w$, respectively,  we get \eqref{funzioni}.  
\par
As for the estimate \eqref{gradienti}, by applying the hidden convexity inequality \eqref{eq:hc}, we have
\[ 
\int_\Omega |\nabla \sigma_{n}^t|^p\,dx\le (1-t)\, \int_\Omega |\nabla v_n|^p\,dx + t\, \int_\Omega |\nabla w_n|^p\,dx, 
\]  
and the right-hand side is bounded, uniformly with respect to $n$. The proof is over.
\end{proof}

\section{Some properties of the energy functional}
\label{sec:3}

We recall the notion of {\it superminimum} and {\it subminimum} for the energy functional 
\[
\mathfrak{F}_{q,\alpha}(\psi):= \frac{1}{p} \int_\Omega |\nabla \psi|^{p} \, dx - \frac{\alpha}{q} \int_\Omega |\psi|^{q} \, dx, \qquad \mbox{ for every } \psi\in X^{q,p}(\Omega),
\]
which is naturally attached to our Lane-Emden equation \eqref{LE}.
\begin{defn}
Let $\alpha>0$, $1 < p < \infty$ and $1 \le q < p$. For every $v \in X^{q,p}(\Omega)$, we introduce the classes
\[
\mathcal{A}^+(v)=\Big\{\psi\in X^{q,p}(\Omega)\, :\, \psi\ge v \mbox{ on } \Omega,\, \psi-v\in X^{q,p}_0(\Omega)\Big\},
\] 
and
\[
\mathcal{A}^-(v)=\Big\{\psi\in X^{q,p}(\Omega)\, :\, \psi\le v \mbox{ on } \Omega,\, \psi-v\in X^{q,p}_0(\Omega)\Big\}.
\]
Then we say that $v$ is a:
\begin{itemize}
\item {\it superminimum for} $\mathfrak{F}_{q,\alpha}$ if
\[
\mathfrak{F}_{q,\alpha}(\psi) \ge \mathfrak{F}_{q,\alpha}(v),\qquad \mbox{ for every } \psi \in\mathcal{A}^+(v);
\]
\item {\it subminimum for} $\mathfrak{F}_{q,\alpha}$ if 
\[
\mathfrak{F}_{q,\alpha}(\psi) \ge \mathfrak{F}_{q,\alpha}(v),\qquad \mbox{ for every } \psi \in\mathcal{A}^-(v);
\]
\item {\it minimum for} $\mathfrak{F}_{q,\alpha}$ if 
\[
\mathfrak{F}_{q,\alpha}(\psi) \ge \mathfrak{F}_{q,\alpha}(v),\qquad \mbox{ for every } \psi\in X^{q,p}(\Omega) \mbox{ such that } \psi-v\in X^{q,p}_0(\Omega).
\]
\end{itemize}
\end{defn} 
\begin{rem}
It is a routine fact to show that a function $v$ is a minimum for $\mathfrak{F}_{q,\alpha}$ if and only if it is both a superminimum and a subminimum.
\end{rem}

We start with the following existence result, which holds under minimal assumptions.
\begin{thm}[Existence]
\label{teo:existence}
Let $\alpha>0$ and $1\le q<p<\infty$. Let $\Omega\subset \mathbb{R}^N$ be an open set which is $q-$admissible. Then for every $U\in X^{q,p}(\Omega)$ the following problem
\[
\inf_{\psi\in X^{q,p}(\Omega)}\Big\{\mathfrak{F}_{q,\alpha}(\psi)\, :\, \psi-U\in W^{1,p}_0(\Omega)\Big\},
\]
admits a solution. Moreover, we have:
\begin{itemize}
\item for $1<q<p$, each minimizer is a solution of the Lane-Emden equation \eqref{LE};
\vskip.2cm
\item for $q=1$, each minimizer $u$ satisfies
\[
-\alpha\le -\Delta_p u\le \alpha,\qquad \mbox{ in }\Omega,
\]
in weak sense;
\vskip.2cm
\item for $q=1$, each non-negative minimizer $u$ (provided it exists) satisfies
\[
-\Delta_p u= \alpha,\qquad \mbox{ in }\Omega,
\]
in weak sense.
\end{itemize}
\end{thm}
\begin{proof}
We first observe that the class of admissible functions is not empty, since the function $U$ itself is admissible. We will use the Direct Method in the Calculus of Variations. At this aim, let us first prove that the infimum above is finite. For every admissible $\psi$, we have 
\begin{equation}
\label{stimainiziale}
\begin{split}
\int_\Omega |\psi|^q\,dx& \le2^{q-1}\,\int_\Omega |\psi-U|^q\,dx+2^{q-1}\,\int_\Omega |U|^q\,dx\\
&\le \frac{2^{q-1}}{\left(\lambda_{p,q}(\Omega)\right)^\frac{q}{p}}\,\|\nabla \psi-\nabla U\|_{L^p(\Omega;\mathbb{R}^N)}^q+2^{q-1}\,\int_\Omega |U|^q\,dx,
\end{split}
\end{equation}
where we used the assumptions that $\psi-U\in W^{1,p}_0(\Omega)$ and that $\Omega$ is $q-$admissible. We now observe that 
\[
\begin{split}
\|\nabla \psi-\nabla U\|_{L^p(\Omega;\mathbb{R}^N)}^q&\le 2^{q-1}\,\|\nabla \psi\|_{L^p(\Omega;\mathbb{R}^N)}^q+2^{q-1}\,\|\nabla U\|_{L^p(\Omega;\mathbb{R}^N)}^q.
\end{split}
\]
By inserting this in \eqref{stimainiziale}, we finally get
\begin{equation}
\label{normalimita}
\int_\Omega |\psi|^q\,dx\le C_1\,\|\nabla \psi\|_{L^p(\Omega;\mathbb{R}^N)}^q+C_2\,\|U\|_{X^{q,p}(\Omega)}^q,
\end{equation}
for some $C_1=C_1(N,p,q,\Omega)>0$ and $C_2=C_2(N,p,q,\Omega)>0$. This yields
\[
\begin{split}
\mathfrak{F}_{q,\alpha}(\psi)&=\frac{1}{p}\,\int_\Omega |\nabla \psi|^p\,dx-\frac{\alpha}{q}\,\int_\Omega |\psi|^q\,dx\\
&\ge \frac{1}{p}\,\int_\Omega |\nabla \psi|^p\,dx-C_1\,\frac{\alpha}{q}\,\left(\int_\Omega |\nabla \psi|^p\,dx\right)^\frac{q}{p}-C_2\,\frac{\alpha}{q}\,\|U\|_{X^{q,p}(\Omega)}^q.
\end{split}
\]
By a suitable application of the generalized Young's inequality
\[
a\,b\le \delta\,\frac{q}{p}\,a^\frac{p}{q}+\delta^{-\frac{q}{p-q}}\,\frac{p-q}{p}\,b^\frac{p}{p-q},\qquad \mbox{ for }a,b\ge 0 \mbox{ and } \delta>0,
\]
we then obtain
\[
\mathfrak{F}_{q,\alpha}(\psi)\ge \frac{1}{p}\,\left(1-\delta\right)\,\int_\Omega |\nabla \psi|^p\,dx-\delta^{-\frac{q}{p-q}}\,\frac{p-q}{p\,q}\,(C_1\,\alpha)^\frac{p}{p-q}-C_2\,\frac{\alpha}{q}\,\|U\|_{X^{q,p}(\Omega)}^q.
\]
If we now choose $\delta=1/2$, we finally end up with 
\begin{equation}
\label{coerciva}
\mathfrak{F}_{q,\alpha}(\psi)\ge \frac{1}{2\,p}\,\int_\Omega |\nabla \psi|^p\,dx-M,
\end{equation}
where $M=M(N,p,q,\alpha,\Omega,U)>0$. This in particular shows that the infimum of $\mathfrak{F}_{q,\alpha}$ over the claimed set is finite.
\par
Let us call $m$ this infimum and consider a sequence $\{\psi_n\}_{n\in\mathbb{N}}$ of admissible functions, such that
\[
\mathfrak{F}_{q,\alpha}(\psi_n)\le m+\frac{1}{n+1},\qquad \mbox{ for every } n\in\mathbb{N}.
\]
By using the estimates \eqref{normalimita} and \eqref{coerciva}, we can then infer that 
\[
\sup_{n\in\mathbb{N}} \|\psi_n\|_{X^{q,p}(\Omega)}<+\infty.
\]
Thus, we get that $\psi_n$ weakly converges in $L^q(\Omega)$ to $\psi$ and $\nabla \psi_n$ weakly converges in $L^p(\Omega;\mathbb{R}^N)$ to a vector filed $\phi$, up to a subsequence. It is easily seen, by the definition of weak gradient, that it must result $\phi=\nabla \psi$. This shows that $\psi\in X^{q,p}(\Omega)$. 
\par
We still need to show that $\psi-U\in W^{1,p}_0(\Omega)$. In order to prove this, we observe that the sequence $\{\psi_n-U\}_{n\in\mathbb{N}}\subset W^{1,p}_0(\Omega)$ is bounded in the norm $X^{q,p}(\Omega)$. However, thanks to Proposition \ref{prop:spazi0}, we know that $W^{1,p}_0(\Omega)=X^{q,p}_0(\Omega)$. This in turn permits to show that $\{\psi_n-U\}_{n\in\mathbb{N}}$ is bounded in $W^{1,p}_0(\Omega)$. Thus $\psi_n-U$ weakly converges in $W^{1,p}_0(\Omega)$, up to a subsequence. By uniqueness, such a limit must coincide with $\psi-U$ and belongs to $W^{1,p}_0(\Omega)$, since the latter is weakly closed.
Moreover, by using again that $\Omega$ is $q-$admissible, we can infer
\[
\lim_{n\to\infty} \|\psi_n-U\|_{L^q(\Omega)}=\|\psi-U\|_{L^q(\Omega)},
\]
up to a subsequence, thanks to Proposition \ref{prop:spazi0}. Thus $\{\psi_n\}_{n\in\mathbb{N}}$ is actually strongly converging in $L^q(\Omega)$. This fact and the weak lower semicontinuity of the $p-$Dirichlet integral show that
\[
\mathfrak{F}_{q,\alpha}(\psi)\le \liminf_{n\to \infty}\mathfrak{F}_{q,\alpha}(\psi_n).
\]
By construction of the sequence $\{\psi_n\}_{n\in\mathbb{N}}$, the function $\psi$ is the desired minimizer. 
\vskip.2cm\noindent
The last part of the statement is now standard for $1<q<p$, it is sufficient to observe that the functional is Gateaux differentiable and compute its first variation. 
\par
In the limit case $q=1$ we must be more careful. Let us take $u$ a minimizer. For every $t>0$ and $\psi\in C^\infty_0(\Omega)$ non-negative, we have 
\[
\mathfrak{F}_{1,\alpha}(u+t\,\psi)\ge \mathfrak{F}_{1,\alpha}(u).
\]
By using the definition of $\mathfrak{F}_{1,\alpha}$ and the triangle inequality for the absolute valute, this implies
\[
\frac{1}{p}\,\int_\Omega \frac{|\nabla u+t\,\nabla \psi|^p-|\nabla u|^p}{t}\,dx+\alpha\,\int_\Omega \psi\,dx\ge 0,
\]
where we further divided by $t>0$. By taking the limit as $t$ goes to $0$, we get
\[
\int_\Omega \langle |\nabla u|^{p-2}\,\nabla u,\nabla \psi\rangle\,dx\ge -\alpha\,\int_\Omega \psi\,dx,\qquad \mbox{ for every } \psi\in C^\infty_0(\Omega),\ \psi\ge 0.
\]
This exactly means that $-\Delta_p u\ge -\alpha$, in weak sense. The other differential inequality is obtained in the same way, by taking this time $t<0$. This is left to the reader.
\par
Finally, still in the case $q=1$, let us suppose that $u$ is a non-negative minimizer. Then, if we define the {\it convex} functional 
\[
\widetilde{\mathfrak{F}}_{1,\alpha}(\psi):=\frac{1}{p}\,\int_\Omega |\nabla \psi|^p\,dx-\alpha\,\int_\Omega \psi\,dx,
\]
we get that $u$ minimizes this functional, as well. It is sufficient to observe that
\[
\widetilde{\mathfrak{F}}_{1,\alpha}(\psi)\ge \mathfrak{F}_{1,\alpha}(\psi)\ge \mathfrak{F}_{1,\alpha}(u)=\widetilde{\mathfrak{F}}_{1,\alpha}(u).
\]
for every admissible $\psi$. This new functional is Gateaux differentiable (because the lower order term is now linear) and its critical points exactly satisfy the weak formulation of the equation $-\Delta_p u=\alpha$. This concludes the proof.
\end{proof}
\begin{rem}
\label{rem:zeroalbordo}
Recall that, thanks to Proposition \ref{prop:spazi0}, in the previous result one could equivalently write the condition $\psi-U\in W^{1,p}_0(\Omega)$ as 
\[
\psi-U\in X^{q,p}_0(\Omega).
\]
Moreover, in the particular case $U\in W^{1,p}_0(\Omega)$, under the previous assumptions we get 
\[
\inf_{\psi\in X^{q,p}(\Omega)}\Big\{\mathfrak{F}_{q,\alpha}(\psi)\, :\, \psi-U\in W^{1,p}_0(\Omega)\Big\}=
\inf_{\psi\in W^{1,p}_0(\Omega)}\mathfrak{F}_{q,\alpha}(\psi),
\]
still thanks to Proposition \ref{prop:spazi0}.
\end{rem}

The hidden convex structure of the functional $\mathfrak{F}_{q,\alpha}$ permits to establish 
an equivalence between supersolutions/subsolutions of \eqref{LE} and superminima/subminima of the functional $\mathfrak{F}_{q,\alpha}$. More precisely, we have the following
\begin{prop}
\label{prop:solmin}
Let $\alpha>0$ and $1\le q<p<\infty$. Let $\Omega \subset \mathbb{R}^N$ be an open set and let $v \in X^{q,p}(\Omega)$ be a positive function. The following facts hold: 
\begin{enumerate}
\item $v$ is a weak supersolution of \eqref{LE} if and only if it is a superminimum for $\mathfrak{F}_{q,\alpha}$;
\vskip.2cm
\item $v$ is a weak subsolution of \eqref{LE} if and only if it is a subminimum for $\mathfrak{F}_{q,\alpha}$.    
\end{enumerate}
\end{prop}
\begin{proof}
We focus on proving the first statement. The proof of the second one runs similarly and it is thus left to the reader. Moreover, the fact that a superminimum is a weak supersolution, it follows in a standard way, by computing the first variation of the functional. We can thus reduce to prove the converse implication.
\vskip.2cm\noindent
We suppose that $v$ is a positive supersolution of \eqref{LE}. Let $\psi\in \mathcal{A}^+(v)$, our aim is to show that 
\begin{equation}
\label{sopra}
\mathfrak{F}_{q,\alpha}(\psi)\ge \mathfrak{F}_{q,\alpha}(v).
\end{equation}
We first observe that if we set 
\[
\widetilde v=\alpha^{-\frac{1}{p-q}}\,v,
\]
then by Remark \ref{rem:scalings} this is a supersolution of the Lane-Emden equation with $\alpha=1$. Moreover, we have 
\[
\mathfrak{F}_{q,1}(\widetilde v)=\alpha^{-\frac{p}{p-q}}\,\mathfrak{F}_{q,\alpha}(v).
\]
Analogously, if $\psi\in \mathcal{A}^+(v)$ and we set
\[
\widetilde \psi=\alpha^{-\frac{1}{p-q}}\,\psi,
\]
then  $\widetilde \psi\in \mathcal{A}^+(\widetilde v)$ and 
\[
\mathfrak{F}_{q,1}(\widetilde \psi)=\alpha^{-\frac{p}{p-q}}\,\mathfrak{F}_{q,\alpha}(\psi).
\]
Thus, in order to prove \eqref{sopra}, we can reduce to prove that 
\begin{equation}
\label{sopra1}
\mathfrak{F}_{q,1}(\widetilde\psi)\ge \mathfrak{F}_{q,1}(\widetilde v).
\end{equation}
We now build the curve $\sigma^t=((1-t)\,\widetilde{v}^q+t\, {\widetilde \psi}^q)^{1/q}$ and observe that
\[
\frac{\sigma^t - \widetilde v}{t} \in X^{q,p}_0(\Omega),
\] 
thanks to Proposition \ref{prop:datumU}. Furthermore, by monotonicity, we observe that 
\[ 
(\sigma^t)^q = (1-t)\, \widetilde v ^q + t \, {\widetilde \psi}^q = \widetilde v^q + t\, (\widetilde \psi^q-\widetilde v^q) \ge \widetilde v^q,
\]
since $\widetilde\psi \ge \widetilde v$. This shows that 
\[
\sigma^t  \in \mathcal{A}^+(\widetilde v).
\]
From \eqref{eq:hc}, we get 
\[
\int_{\Omega} | \nabla \sigma^t |^p \, dx \le (1-t)\, \int_{\Omega} | \nabla \widetilde v |^p \, dx + t \,\int_{\Omega} | \nabla\widetilde\psi |^p \, dx, \qquad \text{ for every } t \in [0,1], 
\]
and also, by construction, we have 
\[ 
\int_{\Omega}  (\sigma^t)^q\,dx = (1-t)\, \int_{\Omega} \widetilde v^q + t\, \int_{\Omega} \widetilde \psi^q\,dx,\qquad \text{ for every } t \in [0,1].
\]
Thus in particular we obtain
\begin{equation}\label{eq:prop1}
\frac{\mathfrak{F}_{q,1}(\sigma^t)-\mathfrak{F}_{q,1}(\widetilde v)}{t} \le \mathfrak{F}_{q,1}(\widetilde\psi) - \mathfrak{F}_{q,1}(\widetilde v), \qquad \text{ for every } t \in (0,1].
\end{equation}
We focus on the term on the left-hand side. We use the ``above tangent'' inequality
\[
\frac{1}{p}\, |z|^{p} \ge \frac{1}{p}\, |z_0|^{p}+ \langle |z_0|^{p-2} z_0, z - z_0 \rangle, \quad \text{ for every } z_0, z \in \mathbb{R}^N, 
\]
with the choices $z= \nabla \sigma^t$ and $z_0=\nabla \widetilde v$.
Thus we get
\[
\begin{split}
\frac{\mathfrak{F}_{q,1}(\sigma^t)-\mathfrak{F}_{q,1}(\widetilde v)}{t} &= \frac{1}{p}\, \int_{\Omega} \frac{|\nabla\sigma^t|^{p}-|\nabla \widetilde v|^{p}}{t} \, dx - \frac{1}{q} \,\int_{\Omega} \frac{(\sigma^t)^q-\widetilde v^q}{t} \, dx \\
&\ge \int_{\Omega} \left\langle |\nabla \widetilde v|^{p-2}\, \nabla \widetilde v, \frac{\nabla \sigma^t - \nabla \widetilde v}{t} \right\rangle \, dx - \frac{1}{q}\, \int_{\Omega} \frac{\widetilde v^q+t\,(\widetilde \psi^q-\widetilde v^q)-\widetilde v^q}{t} \, dx \\
&\ge \int_{\Omega} \widetilde v^{q-1}\, \frac{\sigma^t - \widetilde v}{t} \, dx - \frac{1}{q} \int_\Omega ( \widetilde \psi^q - \widetilde v^q ) \, dx, 
\end{split}
\]
where in the last inequality we used that $\widetilde v$ is a supersolution and the fact that $\sigma^t-\widetilde v$ is a feasible test function. Moreover, we observe that almost everywhere in $\Omega$ it holds
\[
\lim_{t \to 0^+} \frac{\sigma^t - \widetilde v}{t}=\frac{d}{dt} \sigma^t|_{t=0} = \frac{1}{q}\, \widetilde v^{1-q}\, (\widetilde \psi^q-\widetilde v^q).
\]
Then we can pass to the limit as $t$ goes to $0$: from \eqref{eq:prop1} and the estimate above, we obtain 
\[
\mathfrak{F}_{q,1}(\widetilde \psi)-\mathfrak{F}_{q,1}(\widetilde v) \ge \frac{1}{q}\, \int_{\Omega} \widetilde v^{q-1}\, \widetilde v^{1-q}\, (\widetilde \psi^q-\widetilde v^q) \, dx - \frac{1}{q}\, \int_{\Omega} (\widetilde\psi^q-\widetilde v^q) \, dx = 0,
\]
thanks to Fatou's Lemma and to the fact that $v>0$ almost everywhere in $\Omega$, by assumption. This finally established \eqref{sopra1}, as desired.
\end{proof}

\begin{rem}
\label{rem:unicasolpos}
Of course, the previous result implies that a positive function $v \in X^{q,p}(\Omega)$ is a weak solution of \eqref{LE} if and only if it is a solution of
\[
\min_{\psi\in X^{q,p}(\Omega)}\Big\{\mathfrak{F}_{q,\alpha}(\psi)\,:\, \psi-v\in W^{1,p}_0(\Omega)\Big\},
\]
i.e. if it minimizes $\mathfrak{F}_{q,\alpha}$ with respect to its own boundary datum.
\end{rem}

%%%%%%%%%%%%%%%%%%%%%%%%%%%%%%%%%%%%%%%%%%%%%%%%%%%%%%%%%%%%%%%%%%%%%%%%%%%%%%

\section{A comparison principle}
\label{sec:4}

The main tool of this paper is the following comparison principle for positive supersolutions and subsolutions of the Lane-Emden equation \eqref{LE}. This is proved under minimal assumptions, both on the set and the functions.

\begin{thm}[Comparison principle]
\label{thm:comparison}
Let $\alpha>0$ and $1\leq  q < p<\infty$. Let $\Omega \subset \mathbb{R}^N$ be an open connected set. Assume that $u, v \in X^{q,p}(\Omega)$ are two positive functions, such that $u$ is a subsolution and $v$ is a supersolution of \eqref{LE}.  If $(u-v)_+ \in X^{q,p}_0(\Omega)$, then 
\[
v \ge u,\qquad\mbox{ a.\,e. in }\Omega.
\]  
\end{thm}
\begin{proof} 
We first observe that for $q=1$, the result is well-known. It can be obtained with exactly the same proof of the comparison principle for $p-$harmonic functions (see for example \cite[Theorem 2.15]{Lin2}). For completeness, we sketch the argument: it is sufficient to take the test function $\psi=(u-v)_+ \in X^{q,p}_0(\Omega)$ in the weak formulations for $u$ and $v$. This gives
\[
\int_{\Omega} \langle |\nabla u|^{p-2} \nabla u, \nabla (u-v)_+ \rangle \, dx \le \alpha\int_{\Omega} (u-v)_+ \, dx,
\]
and
\[
\int_{\Omega} \langle |\nabla v|^{p-2} \nabla v, \nabla (u-v)_+ \rangle \, dx \ge \alpha\int_{\Omega} (u-v)_+ \, dx. 
\]
By subtracting them, we obtain
\[
\int_{\{u>v\}} \langle |\nabla u|^{p-2} \nabla u-|\nabla v|^{p-2}\,\nabla v, \nabla u-\nabla v\rangle \, dx\le 0.
\]
If we now use that the vector field $z\mapsto |z|^{p-2}\,z$ is strictly monotone\footnote{In other words, we have 
\[
\langle |z|^{p-2}\,z-|w|^{p-2}\,w,z-w\rangle \ge 0, 
\]
and the equality holds if and only if $z=w$.
This property is a consequence of the strict convexity of the map $z\mapsto |z|^p/p$, whose gradient is precisely $z\mapsto |z|^{p-2}\,z$.}, we get the desired conclusion with standard arguments. We leave the details to the reader.  
\vskip.2cm\noindent	 
Let us now focus on the case $1<q<p$. With Proposition \ref{prop:solmin} and the hidden convexity property at hand, we can essentially reproduce the proof of the classical comparison principle for {\it strictly convex} integral functionals of the Calculus of Variations, see for example \cite[Lemma 1.1]{Gi}. The identification of equality cases in the hidden convexity property will play a crucial role.
\par
We define $\varphi=\min\{v,u\}$ and observe that it has the following properties 
\[
\varphi \in X^{q,p}(\Omega),\qquad u-\varphi=(u-v)_+ \in X^{q,p}_0(\Omega)\qquad \mbox{ and  }\qquad\varphi \le u \mbox{ in } \Omega,
\] 
so that it belongs to $\mathcal{A}^-(u)$.
Observe that, since $u$ is a positive subsolution of \eqref{LE}, from Proposition \ref{prop:solmin} it is a subminimum for $\mathfrak{F}_{q,\alpha}$, as well. By using this and the properties of $\varphi$, we get 
\begin{equation}\label{eq:cp1}
\mathfrak{F}_{q,\alpha}(\varphi) \ge \mathfrak{F}_{q,\alpha}(u).
\end{equation}
By recalling that the weak gradient of $\varphi$ is given by (see \cite[Corollary 6.18]{LL})
\[
\nabla \varphi =
\left\{\begin{array}{ll}
\nabla v, &\text{ a.\,e. on } \{v < u\}, \\
\nabla u, &\text{ a.\,e. on } \{u < v\}, \\
\nabla v = \nabla u, & \text{ a.\,e. on } \{v=u\},
\end{array}
\right.
\]
inequality \eqref{eq:cp1} entails that
\[
\frac{1}{p}\, \int_{\{v<u\}} |\nabla v|^{p} \, dx - \frac{\alpha}{q}\, \int_{\{v<u\}} |v|^q \, dx \ge \frac{1}{p}\, \int_{\{v<u\}} |\nabla u|^{p} \, dx - \frac{\alpha}{q}\, \int_{\{v<u\}} |u|^q \, dx. 
\]
If we add the quantity 
\[
\frac{1}{p} \int_{\{u \le v\}} |\nabla v|^{p} \, dx - \frac{\alpha}{q} \int_{\{u \le v\}} |v|^q \, dx,
\] 
on both sides, we end up with
\begin{equation}
\label{eq:cp2}
\mathfrak{F}_{q,\alpha}(v) \ge \mathfrak{F}_{q,\alpha}(\psi),\qquad \mbox{ where } \psi=\max\{u,v\} \in X^{q,p}(\Omega).
\end{equation}
Observe that $\psi$ has the following properties 
\[
\psi-v=(u-v)_+ \in X^{q,p}_0(\Omega)\qquad \mbox{ and } \qquad v \le \psi \mbox{ in }\Omega,
\]
i.\,e. it belongs to $\mathcal{A}^+(v)$.
Since $v$ is a positive supersolution of \eqref{LE}, it is a superminimum of $\mathfrak{F}_{q,\alpha}$ (again thanks to Proposition \ref{prop:solmin}). This remark and the fact that $\psi\in \mathcal{A}^+(v)$ imply that we must have
\[
\mathfrak{F}_{q,\alpha}(v) \le \mathfrak{F}_{q,\alpha}(\psi),
\]
as well. Thus equation \eqref{eq:cp2} must be an identity.
\par

We now set
\[
\sigma=\left(\frac{v^q+\psi^q}{2}\right)^\frac{1}{q} \in X^{q,p}(\Omega),
\] 
and use the hidden convexity property \eqref{eq:hc} with $t=1/2$, so to obtain
\begin{equation}
\label{eq:cp3}
\mathfrak{F}_{q,\alpha}(\sigma) \le \frac{1}{2}\, \mathfrak{F}_{q,\alpha}(v) + \frac{1}{2}\, \mathfrak{F}_{q,\alpha}(\psi) = \mathfrak{F}_{q,\alpha}(v).
\end{equation}
Since $\psi-v\in X^{q,p}_0(\Omega)$, by Proposition \ref{prop:datumU} we have $\sigma-v\in X^{q,p}_0(\Omega)$, as well. Moreover, since by construction $\psi\ge v$, we also have $\sigma\ge v$. We can thus test the superminimality of $v$ against $\sigma$ and get that actually also \eqref{eq:cp3} must hold as an equality. In particular, we have
\[
\int_\Omega |\nabla \sigma|^p\,dx=\frac{1}{2}\,\int_\Omega |\nabla v|^p\,dx+\frac{1}{2}\,\int_\Omega |\nabla \psi|^p\,dx.
\]
We use the equality cases of Theorem \ref{thm:shc} with $r=q<p$ and $t=1/2$, so to get that 
\[
\mbox{ either }\qquad \psi=v \qquad \mbox{ or }\qquad \psi \mbox{ and } v \mbox{ are both constant}. 
\]
In the first case we directly get the desired conclusion, since $\psi=\max\{u,v\}$. The second case can not occur, since $v$ is positive and from the equation we easily see that a positive constant can not be a supersolution. This concludes the proof.
\end{proof}
\begin{rem}
The assumption $(u-v)_+\in X^{q,p}_0(\Omega)$ is a weak surrogate of the usual condition
\[
v\ge u \qquad \mbox{ on }\partial\Omega,
\]
appearing in comparison principles. Whenever a trace theory is available (for example, if $\Omega$ is smooth enough), the two conditions coincide.
\end{rem}
From the comparison principle, we get the following uniqueness result. 
\begin{thm}[Uniqueness of minimizers]
\label{teo:uniquemin}
Let $\alpha>0$ and $1\le q<p<\infty$. Let $\Omega \subset \mathbb{R}^N$ be an open connected set, which is $q-$admissible. For every function $U \in X^{q,p}(\Omega)$, the minimization problem
\[
\inf_{\psi\in X^{q,p}(\Omega)} \Big\{\mathfrak{F}_{q,\alpha}(\psi)\, :\, \psi-U\in W^{1,p}_0(\Omega)\Big\},
\]
admits:
\begin{itemize}
\item[{\it (i)}] exactly one solution, when $U\in X^{q,p}(\Omega)\setminus W^{1,p}_0(\Omega)$ is non-negative. Moreover, such a solution is positive;
\vskip.2cm
\item[{\it (ii)}] exactly two solutions, when $U\in W^{1,p}_0(\Omega)$. In this case, both solutions have constant sign and they coincide, up to the choice of the sign.
\end{itemize}
\end{thm}

\begin{proof}
Observe that we already know that the minimization problem does admit a solution, by virtue of Theorem \ref{teo:existence}. We will also rely on the optimality conditions, contained in the same result.
\vskip.2cm\noindent

\noindent First of all we note that   a minimizer $u$ can not identically vanish.
Indeed this is trivial  in the case $(i)$ since the null function is not admissible for the minimization problem.
In the case $(ii)$ it is sufficient to take a function $\psi \in W_0^{1,p}(\Omega)\setminus\{0\}$ and $0 < t \ll 1$ to get \[ 
\mathfrak{F}_{q,\alpha}(t\, \psi) = \frac{t^p}{p} \int_{\Omega} |\nabla \psi|^p \, dx - \frac{\alpha\,t^q}{q} \int_{\Omega} |\psi|^q \, dx < 0 = \mathfrak{F}_{q,\alpha}(0), 
\]
since $q<p$. 
\par
Furthermore, we observe that, in both cases $(i)$ and $(ii)$, if $u$ is a minimizer then $|u|$ is a minimizer, as well. Indeed,  $|u|-U\in W^{1,p}_0(\Omega)$ thanks to Lemma \ref{lem:Upos} and to the fact that $U$ is non-negative. 
Hence, since $\mathfrak{F}_{q,\alpha}$ is even and $|u|$ is still admissible,
we obtain that $|u|$ is a minimizer, as well. By minimality for both $u$ and $|u|$,  we deduce that $u_+$ is weakly $p-$superharmonic.
 Indeed, when $1<q<p$ we  have 
\[
\int_\Omega \langle |\nabla u|^{p-2}\,\nabla u,\nabla \psi\rangle\,dx=\alpha\,\int_\Omega |u|^{q-2}\,u\,\psi\,dx,\qquad\mbox{ for every } \psi\in W^{1,p}_0(\Omega),
\]
and
\[
\int_\Omega \langle |\nabla |u||^{p-2}\,\nabla |u|,\nabla \psi\rangle\,dx=\alpha\,\int_\Omega |u|^{q-1}\,\psi\,dx,\qquad\mbox{ for every } \psi\in W^{1,p}_0(\Omega).
\]
We can sum up these two integral identities: by observing that 
\[
\frac{|u|^{q-2}\,u+|u|^{q-1}}{2}=(u_+)^{q-1},\qquad \mbox{ a.\,e. in }\Omega,
\]
and\footnote{It is sufficient to use the following classical fact from the theory of Sobolev spaces
\[
\nabla |u|=\left\{\begin{array}{cc}
\nabla u,& \mbox{ a.\,e on } \{u>0\},\\
-\nabla u,& \mbox{ a.\,e on } \{u<0\},\\
0,& \mbox{ a.\,e on } \{u=0\},
\end{array}
\right.
\]
see for example \cite[Theorem 6.17]{LL}.}
\[
\frac{|\nabla u|^{p-2}\,\nabla u+|\nabla |u||^{p-2}\,\nabla |u|}{2}=|\nabla u_+|^{p-2}\,\nabla u_+,\qquad \mbox{ a.\,e. in }\Omega,
\]
we then obtain
\[
\int_\Omega \langle |\nabla u_+|^{p-2}\,\nabla u_+,\nabla \psi\rangle\,dx=\alpha\,\int_\Omega (u_+)^{q-1}\,\psi\,dx,\qquad\mbox{ for every } \psi\in W^{1,p}_0(\Omega).
\]
In particular,  $u_+$ is a weakly $p-$superharmonic function on $\Omega$. 
In the case $q=1$, a little additional care is needed. From Theorem \ref{teo:existence}, we know that 
\[
\int_\Omega \langle |\nabla u|^{p-2}\,\nabla u,\nabla \psi\rangle\,dx\ge -\alpha\,\int_\Omega \psi\,dx,\qquad\mbox{ for every } \psi\in W^{1,p}_0(\Omega),\ \psi\ge 0,
\]
and
\[
\int_\Omega \langle |\nabla |u||^{p-2}\,\nabla |u|,\nabla \psi\rangle\,dx=\alpha\,\int_\Omega \psi\,dx,\qquad\mbox{ for every } \psi\in W^{1,p}_0(\Omega).
\]
We can sum up the two relations as above, when testing with a non-negative $\psi$: the terms containing $\alpha$ cancel and
 we now directly get that $u_+$ is weakly $p-$superharmonic.
\vskip.2cm\noindent 
We can now prove uniqueness of the minimizer. Let us start with case $(i)$. In this case  any minimizer must be positive. Indeed, let $u$ be a minimizer. The assumption $u-U\in W^{1,p}_0(\Omega)$ and the fact that $U\in X^{q,p}(\Omega)\setminus W^{1,p}_0(\Omega)$ is non-negative entail that we must have
\[
u_+\not\equiv 0.
\]
Indeed, if $u_+$ would identically vanish, we would have $u=-u_-$ and $|u|=u_-$. By Lemma \ref{lem:Upos} we know that $|u|-U\in W^{1,p}_0(\Omega)$. Thus we would obtain
\[
U=\frac{(U-u)+(U-|u|)}{2}\in W^{1,p}_0(\Omega),
\]
which contradicts the assumption on $U$.
\par
Since $u_+$ is weakly $p-$superharmonic on the connected set $\Omega$, the minimum principle  implies that we must have $u_+>0$ almost everywhere in $\Omega$. In particular, we get
\[
u=u_+>0\qquad \mbox{ in }\Omega,
\]
as desired.
\par
Now we assume  by contradiction that  the minimization problem above admits two distinct minimizers $v,u \in X^{q,p}(\Omega)$. From the discussion above, $v$ and $u$ are positive functions.
Of course, we have
\[
(u-v)_+\in W^{1,p}_0(\Omega)\qquad \mbox{ and }\qquad (v-u)_+\in W^{1,p}_0(\Omega).
\]
We can then apply Theorem \ref{thm:comparison}, first by considering $u$ as a subsolution of \eqref{LE} and $v$ as a supersolution, then the other way round. In conclusion we get
\[
v\ge u\quad \mbox{ a.\,e. in } \Omega\qquad \mbox{ and }\qquad u\ge v\quad \mbox{ a.\,e. in } \Omega.
\] 
This implies that $v$ and $u$ must coincide in $\Omega$, thus obtaining a contradiction.
\vskip.2cm\noindent
We now focus on case $(ii)$, which is slightly subtler. In this case, we first prove that each minimizer is either positive or negative.
Indeed, let $u$ be a minimizer, we write 
\[
u=u_+-u_-\qquad \mbox{ and observe that }\qquad u_+,u_-\in W^{1,p}_0(\Omega).
\]
 If $u_+\not\equiv 0$  then  the same argument as above shows  that $u=u_+>0$ almost everywhere in $\Omega$. If on the contrary $u_+\equiv 0$, we get $u=-u_-$. Since the functional $\mathfrak{F}_{q,\alpha}$ is even, we get that $-u=u_-$ is still a minimizer. It is actually a non-negative minimizer, thus it solves the relevant Euler-Lagrange equation, by Theorem \ref{teo:existence}. In particular, $u_-$ is a non-negative weakly $p-$superharmonic function, which is not identically vanishing. Again by the minimum principle, we get that $u_-$ must be positive on $\Omega$ and thus the desired conclusion follows.
\par
Finally, we assume to have two distinct minimizers $v,u \in W^{1,p}_0(\Omega)$. The previous considerations and Theorem \ref{teo:existence} show that $v$ and $u$ are both constant sign  solutions of \eqref{LE}, not identically vanishing. If we assume that they are both positive, by using Theorem \ref{thm:comparison} as in the first part of the proof, we get again that $u=v$ in $\Omega$, which is not possible. Similarly, if both are negative, then $-u$ and $-v$ are positive solutions and again we get that they must coincide. The only possibility left is thus that $u$ is positive and $v$ is negative: in this case, we can apply the comparison principle as before, to the pair $u$ and $-v$. This finally gives that we must have
\[
u=-v,\qquad \mbox{ a.\,e. in }\Omega.
\]
The proof is now over.
\end{proof}
By joining the previous result and Remark \ref{rem:unicasolpos}, we get the following uniqueness result for the Lane-Emden equation \eqref{LE}.
\begin{cor}
\label{coro:unicasolpos}
Let $\alpha>0$ and $1\le q<p<\infty$. Let $\Omega \subset \mathbb{R}^N$ be an open connected set, which is $q-$admissible. For every $U \in X^{q,p}(\Omega)$ non-negative, the boundary value problem 
\[
\left\{\begin{array}{l}
-\Delta_{p} u = \alpha\, |u|^{q-2}\,u, \mbox{ in } \Omega,\\
u-U\in W^{1,p}_0(\Omega),\\
u>0, \mbox{ in }\Omega,
\end{array}
\right.
\]
admits a unique solution.
\end{cor}

\begin{defn}
\label{defn:w}
Let $\alpha>0$ and $1\le q<p<\infty$. Let $\Omega\subset \mathbb{R}^N$ be an open connected set, which is $q-$admissible. We will indicate by $w_{\Omega,\alpha}\in W^{1,p}_0(\Omega)$ the unique positive solution of 
\[
\min_{\psi\in W^{1,p}_0(\Omega)} \mathfrak{F}_{q,\alpha}(\psi).
\]
In light of Theorem \ref{teo:uniquemin}, such a definition is well-posed. We also observe that, by Corollary \ref{coro:unicasolpos}, such a function is also the unique positive solution of \eqref{LE} with homogeneous Dirichlet boundary conditions. In the case $\alpha=1$, we will simply indicate this function by $w_\Omega$. By recalling Remark \ref{rem:scalings}, we have the relation
\begin{equation}
\label{walfa}
w_{\Omega,\alpha}=\alpha^\frac{1}{p-q}\,w_{\Omega}.
\end{equation}
\end{defn}
\begin{rem}[The case $q=1$ in a ball]
\label{rem:case1}
In this case, the function $w_{\Omega,\alpha}$ can be explicitly computed, by recalling that for every $N\ge 1$ we have
\[
w_{B_1}(x)=\frac{p-1}{p}\,N^{-\frac{1}{p-1}}\,\left(1-|x|^\frac{p}{p-1}\right),\qquad \mbox{ for } x\in B_1.
\]
\end{rem}
By combining the comparison principle with standard properties of solutions to elliptic PDEs, we get the following ``hierarchy'' of solutions for the Lane-Emden equation \eqref{LE} with homogeneous Dirichlet boundary conditions. This asserts that all solutions must be ``trapped''  between $w_{\Omega,\alpha}$ and $-w_{\Omega,\alpha}$. 

\begin{cor}[Hierarchy of solutions]
\label{cor:hierarchy}
Let $\alpha>0$ and $1<q<p<\infty$. Let $\Omega \subset \mathbb{R}^N$ be an open connected set, which is $q-$admissible. 
Then for every sign-changing weak solution $v\in W^{1,p}_0(\Omega)$ of \eqref{LE} with homogeneous Dirichlet boundary conditions, we have
\[
|v|\le w_{\Omega,\alpha},\qquad \mbox{ a.\,e. in } \Omega. 
\]
\end{cor}
\begin{proof}
Let $v\in W^{1,p}_0(\Omega)$ be a solution of \eqref{LE}. We claim that $V:=\max\{v,w_{\Omega,\alpha}\}\in W^{1,p}_0(\Omega)$ is a positive weak subsolution of the same equation. If this were true, then we would get from Theorem \ref{thm:comparison} that 
\[
v\le V\le w_{\Omega,\alpha},\qquad \mbox{ a.\,e. in }\Omega.
\]
By repeating the argument with $-v$ (which is still a weak solution of the same equation), we would finally get the desired conclusion.
\par
We are left with proving that $\max\{v,w_{\Omega,\alpha}\}$ is a weak subsolution. This is quite classical, we briefly sketch the argument: for every $n\in\mathbb{N}\setminus\{0\}$, we take 
\[
H_n(t)=\left\{\begin{array}{rl}
0, & \mbox{ if }t\le 0,\\
n\,t,& \mbox{ if } 0\le t\le 1/n,\\
1,& \mbox{ if } t\ge 1/n,
\end{array}
\right.
\]
i.e. this is a Lipschitz approximation of the Heaviside step function. For every $\psi\in C^\infty_0(\Omega)$ non-negative, we then insert in the weak formulations of the equations for $v$ and $w_{\Omega,\alpha}$, the test functions
\[
\varphi=H_n(v-w_{\Omega,\alpha})\,\psi\qquad \mbox{ and }\qquad \varphi=(1-H_n(v-w_{\Omega,\alpha}))\,\psi,
\]
respectively. We thus get
\[
\begin{split}
\int_\Omega \langle |\nabla v|^{p-2}\,\nabla v,\nabla v-\nabla w_{\Omega,\alpha}\rangle\,H_n'(v-w_{\Omega,\alpha})\,\psi\,dx&+\int_\Omega \langle |\nabla v|^{p-2}\,\nabla v,\nabla \psi\rangle\,H_n(v-w_{\Omega,\alpha})\,dx\\
&=\alpha\,\int_\Omega |v|^{q-2}\,v\,H_n(v-w_{\Omega,\alpha})\,\psi\,dx,
\end{split}
\]
and
\[
\begin{split}
-\int_\Omega \langle |\nabla w_{\Omega,\alpha} |^{p-2}\,\nabla w_{\Omega,\alpha},\nabla v-\nabla w_{\Omega,\alpha}\rangle&\,H_n'(v-w_{\Omega,\alpha})\,\psi\,dx\\
&+\int_\Omega \langle |\nabla w_{\Omega,\alpha}|^{p-2}\,\nabla w_{\Omega,\alpha},\nabla \psi\rangle\,(1-H_n(v-w_{\Omega,\alpha}))\,dx\\
&=\alpha\,\int_\Omega w_{\Omega,\alpha}^{q-1}\,(1-H_n(v-w_{\Omega,\alpha}))\,\psi\,dx.
\end{split}
\]
We now sum up these two identities, use that the vector field $z\mapsto |z|^{p-2}\,z$ is monotone and that $H_n$ is non-decreasing. We can thus obtain
\[
\begin{split}
\int_\Omega \langle |\nabla v|^{p-2}\,\nabla v,\nabla \psi\rangle\,H_n(v-w_{\Omega,\alpha})\,dx
&+\int_\Omega \langle |\nabla w_{\Omega,\alpha}|^{p-2}\,\nabla w_{\Omega,\alpha},\nabla \psi\rangle\,(1-H_n(v-w_{\Omega,\alpha}))\,dx\\
&\le \alpha\,\int_\Omega |v|^{q-2}\,v\,H_n(v-w_{\Omega,\alpha})\,\psi\,dx\\
&+\alpha\,\int_\Omega w_{\Omega,\alpha}^{q-1}\,(1-H_n(v-w_{\Omega,\alpha}))\,\psi\,dx.
\end{split}
\]
We can now pass to the limit as $n$ goes to $\infty$, with a straightforward application of the Lebesgue Dominated Convergence Theorem. By observing that for almost every $x\in\Omega$ we have 
\[
\lim_{n\to\infty} H_n(v(x)-w_{\Omega,\alpha}(x))=\left\{\begin{array}{rl}
1,& \mbox{ if } v(x)\ge w_{\Omega,\alpha}(x),\\
0,& \mbox{ otherwise},
\end{array}
\right.
\]
and recalling that (see again \cite[Corollary 6.18]{LL})
\[
\nabla V=\nabla \max\{v,w_{\Omega,\alpha}\} =
\left\{\begin{array}{ll}
\nabla w_{\Omega,\alpha}, &\text{ a.\,e. on } \{v < w_{\Omega,\alpha}\}, \\
\nabla v, &\text{ a.\,e. on } \{w_{\Omega,\alpha} < v\}, \\
\nabla v = \nabla w_{\Omega,\alpha}, & \text{ a.\,e. on } \{v=w_{\Omega,\alpha}\},
\end{array}
\right.
\]
we finally obtain
\[
\int_\Omega \langle |\nabla V |^{p-2}\,\nabla V,\nabla\psi\rangle\,dx\le \alpha\, \int_\Omega V^{q-1}\,\psi\,dx,\quad \text{ for every } \psi \in C_0^{\infty}(\Omega),\ \psi \ge 0.
\]
Thus $V$ is a positive weak subsolution of \eqref{LE} and the proof is over.
\end{proof}

The following result collects some basic properties of the positive solution in the case of an interval. In light of \eqref{walfa}, it is not restrictive to take $\alpha=1$. This will be needed in the next section.
\begin{lemma}[One-dimensional case]
\label{lem:1d}
Let  $1\le q<p<\infty$. If we denote by $I=(-1,1)$, the function $w_I\in W^{1,p}_0(I)$ has the following properties:
\begin{enumerate}
\item it is even, monotone increasing on $(-1,0)$ and monotone decreasing on $(0,1)$;
\vskip.2cm
\item both $w_I$ and $|w_I'|^{p-2}\,w_I'$ belongs to $C^1(\overline{I})$;
\vskip.2cm
\item it is the unique solution of 
\begin{align}\label{eq:mixedproblem}
\left\{\begin{array}{rcll}
-(|w_I'|^{p-2}\, w_I')'&=& w_I^{q-1},&\text{ in } (-1,0), \\
w_I(-1)&=&0, \\
w_I'(0)&=&0;
\end{array}
\right.
\end{align}
\item $w_I'(t)>0$ for every $t\in (-1,0)$;
\vskip.2cm
\item it holds 
\begin{equation}\label{eq:C}
\int_{-1}^0 |w_I|^q \, dt =\left( \frac{2}{\pi_{p,q}} \right)^{\frac{p\,q}{p-q}},
\end{equation}
where $\pi_{p,q}$ is defined in \eqref{pi}.
\end{enumerate}
\end{lemma}
\begin{proof}
We proceed point by point.
\begin{enumerate}
\item The fact that $w_I$ is even follows from its uniqueness. Indeed, if this were not true, the new function $\widetilde w_I(t)=w_I(-t)$ would be another positive minimizer of $\mathfrak{F}_{q,1}$. As for the claimed monotonicity, we observe that the new function
\[
\widehat{w}_I(t)=\left\{\begin{array}{cc}
\displaystyle\int_{-1}^t |w_I'(\tau)|\,d\tau,& \mbox{ for } t\in(-1,0),\\
&\\
\displaystyle\int_{t}^1 |w_I'(\tau)|\,d\tau,& \mbox{ for } t\in(0,1),
\end{array}
\right.
\]
is still admissible, monotone on both subintervals and such that
\[
|\widehat{w}_I'(t)|=|w_I'(t)|\qquad \mbox{ and }\qquad \widehat{w}_I(t)\ge w_I(t),\qquad \mbox{ for a.\,e. }t\in I. 
\]
Thus $\widehat{w}_I$ is still a minimizer and thus, by uniqueness, it must coincide with $w_I$;
\vskip.2cm
\item by minimality, we know that $w_I$ is a weak solution of 
\[
\begin{cases}
-(|w_I'|^{p-2}\, w_I')'= w_I^{q-1} \quad \text{in } I, \\
w_I(-1)=w_I(1)=0. \\
\end{cases}
\]
By \cite[Theorem 3.1]{DM}, we know that such a problem admits a unique positive solution $u\in C^1(\overline{I})$ such that $|u'|^{p-2}\,u'\in C^1(\overline{I})$, as well. Such a solution is also a weak solution and thus, by uniqueness, it must coincide with $w_I$. This proves the claimed regularity properties of $w_I$;
\vskip.2cm
\item this simply follows from the previous point and the symmetry of $w_I$;
\vskip.2cm
\item since   $w_I>0$ and  $w_I' \ge 0$ in $(-1,0)$, by using  \eqref{eq:mixedproblem},  we get that  
\[
-((w_I')^{p-1})' = w_I^{q-1} > 0, \qquad \mbox{ in } (-1,0).
\]
This implies that $(w_I')^{p-1}$ is strictly decreasing and the same is true for $w_I'$, as well. In particular,  
\[
w_I'(t)>w_I'(0)=0,  \qquad \mbox{ for every } t\in (-1,0),
\]
as claimed;
\vskip.2cm
\item finally, in order to prove \eqref{eq:C}, we recall that
\[
\min_{\psi\in W^{1,p}_0(I)}\left\{\frac{1}{p}\,\int_I |\psi'|^p\,dt-\frac{1}{q}\,\int_I |\psi|^q\,dt\right\}=\frac{q-p}{q\,p}\,\left(\frac{1}{\lambda_{p,q}(I)}\right)^{\frac{q}{p-q}}.
\] 
This can be proved by a standard homogenization trick, replacing $\psi$ by $t\,\psi$ and then optimizing in $t>0$.
By recalling that from the optimality condition we have
\[
\int_I |w_I'|^p\,dt=\int_I |w_I|^q\,dt,
\]
and that it holds (simply by scaling)
\[
\lambda_{p,q}(I)=2^\frac{q-p}{q}\,\left(\frac{\pi_{p,q}}{2}\right)^p,
\]
we obtain 
\[
\left(\frac{1}{p}-\frac{1}{q}\right)\,\int_I |w_I|^q\,dt=2\,\frac{q-p}{q\,p}\,\left(\frac{2^p}{(\pi_{p,q})^p}\right)^{\frac{q}{p-q}}.
\]
Thus \eqref{eq:C} follows, by recalling that $w_I$ is even.
\end{enumerate}
This concludes the proof.
\end{proof}

%%%%%%%%%%%%%%%%%%%%%%%%%%%%%%%%%%%%%%%%%%%%%%%%%%%%%%%%%%

\section{Applications to geometric estimates}
\label{sec:5}

\subsection{Solutions of the Lane-Emden equation}
 
The following expedient lemma will be useful. We point out that here the restriction on $q<p$ is not needed.

\begin{lemma}\label{lemma:3.1}
Let $1< p<\infty$ and $1\le q<\infty$. Let $f \in C^1([a,b])$ be a non-negative and non-decreasing function, such that 
\[
|f'|^{p-2} f'\in C^1([a,b]),
\]
and which satisfies
\[
%\left\{\begin{array}{rcll}
-(|f'|^{p-2} f')'=C\, f^{q-1}, \qquad \text{in } [a,b], \\
%f(a)&=&0,&
%\end{array}
%\right.
\]
for some $C>0$.
Let $\Omega \subset \mathbb{R}^N$ be an open set and let $u \in W^{1,p}(\Omega) \cap L^{\infty}(\Omega)$ be a weakly $p-$superharmonic function.
Moreover we assume that $a \le u \le b$ almost everywhere in $\Omega$. Then the composition $\phi=f \circ u$ satisfies 
\[
-\Delta_{p} \phi \ge C\, |\nabla u|^p\, \phi^{q-1},\qquad \mbox{ in }\Omega,
\]
in weak sense. 
\end{lemma}
\begin{proof}
We insert in \eqref{eq:up-sup} the test function 
\[
\psi=(|f'(u)|^{p-2} f'(u))\, \eta,
\] 
with $\eta \in C^{\infty}_0(\Omega)$ a non-negative function. Thanks to the assumptions on $f$ and $u$, the Chain Rule formula ensures that $\psi \in W_0^{1,p}(\Omega) \cap L^{\infty}(\Omega)$ and 
\[ 
\nabla \psi =  (|f'|^{p-2}\, f')'(u) \, \nabla u \, \eta + |f'(u)|^{p-2}\, f'(u) \, \nabla \eta. 
\]
We thus get
\[
\begin{split}
0 &\le - C\, \int_{\Omega} |\nabla u|^p\, f(u)^{q-1} \eta \, dx + \int_\Omega |f'(u)|^{p-2}\,f'(u)\, \langle  |\nabla u|^{p-2}\, \nabla u, \nabla \eta \rangle \, dx,
\end{split}	
\]
where we also used the equation satisfied by $f$. This gives in particular that
\[ 
\int_{\Omega} \langle |\nabla f(u)|^{p-2} \,\nabla f(u), \nabla \eta \rangle \, dx \ge C\, \int_\Omega |\nabla u|^p f(u)^{q-1}  \eta \, dx.
\]
By recalling the definition $\phi=f \circ u$, we can conclude.
\end{proof}

We will use the previous result to construct a special supersolution to the Lane-Emden equation with some geometric contents, in open convex sets. 
\par
%For an open set $\Omega\subset\mathbb{R}^N$, we recall the definition of {\it high ridge set}
%\[
%M(\Omega)=\Big\{x\in\Omega\,:\, B_{r_\Omega}(x)\subset\Omega\Big\},
%\]
%i.e. the collection of centers of maximal balls inscribed in $\Omega$. %We recall that this is a convex set. 
We recall that we denote by $r_\Omega$ the inradius of a set, defined by \eqref{inradius}. We also recall that we indicate by $I=(-1,1)$.
We then have the following
\begin{thm}[Double-sided pointwise estimate] 
\label{thm:bounds}
Let $\alpha>0$ and $1\le q<p<\infty$. Let $\Omega \subset \mathbb{R}^N$ be an open connected set, which is $q-$admissible.
 Then if $B_r(x_0)\subset\Omega$, it holds
\begin{equation}
\label{eq:pointestimate}
w_{B_1,\alpha}\left( \frac{x-x_0}{r} \right) \le r^{-\frac{p}{p-q}}\, \,w_{\Omega,\alpha}(x),\qquad \mbox{ for a.\,e. }x\in\Omega,
\end{equation}
where the function on the left-hand side is extended by zero to the whole $\Omega$.
Moreover, if $\Omega$ is bounded and convex, it also holds
\begin{equation}
\label{eq:pointestimate2}
r_\Omega^{-\frac{p}{p-q}}\, \,w_{\Omega,\alpha}(x) \le w_{I,\alpha}\left(\frac{d_{\Omega}(x)}{r_\Omega}-1\right),\qquad \mbox{ for a.\,e. }x\in\Omega.
\end{equation}
Finally, both estimates are sharp.
\end{thm}
\begin{proof}
We prove separately the lower and upper bounds. In both cases, we heavily rely on the comparison principle of Theorem \ref{thm:comparison}. By recalling \eqref{walfa}, it is sufficient to prove the result for $\alpha=1$.
\vskip.2cm\noindent
{\bf Lower bound.} %Let $x_0\in M(\Omega)$ and 
Let $w_{B_{r}(x_0)}$ be the positive solution of \eqref{LE} in $B_{r}(x_0)\subset \Omega$, with $\alpha=1$. Then, by Remark \ref{rem:scalings} and the uniqueness of the positive solution, we know that
\[
w_{B_{r}(x_0)}(x) =  r^{\frac{p}{p-q}} \,w_{B_1}\left( \frac{x-x_0}{r} \right).
\]
Since $w_{B_{r}(x_0)}\in W^{1,p}_0(B_{r}(x_0))$ and $w_\Omega\ge 0$ on $B_{r}(x_0)$, by Lemma \ref{lem:Upos} part $(ii)$,
 we have that 
 \[
 (w_{B_{r}(x_0)}-w_\Omega)_+\in X^{q,p}_0(B_{r}(x_0))=W^{1,p}_0(B_{r}(x_0)).
 \] 
 Moreover  both $w_{B_{r}(x_0)}$ and $w_\Omega$ are positive solutions to \eqref{LE} in $B_r(x_0)$.
Hence, thanks to Theorem \ref{thm:comparison}, we obtain 
\begin{equation}
\label{eq:lowerboundB}
 r^{\frac{p}{p-q}}\, w_{B_1}\left( \frac{x-x_0}{r} \right)=w_{B_{r}(x_0)}(x) \le w_\Omega (x), \qquad \text{ for a.\,e. }x \in B_r(x_0).
\end{equation} 
Moreover, we can extend $w_{B_{r}(x_0)}$ to the whole $\Omega$ by setting it to be zero in $\Omega \setminus B_{r}(x_0)$. Then \eqref{eq:lowerboundB} holds almost everywhere in $\Omega$.
\vskip.2cm\noindent
{\bf Upper bound.} We define 
\[
u=\frac{d_{\Omega}}{r_\Omega}-1 \in W^{1,p}(\Omega) \cap L^{\infty}(\Omega),
\]
and observe that this is a weakly $p-$superharmonic function. Indeed, since $\Omega$ is convex, the distance function $d_\Omega$ is concave and thus weakly superharmonic (see \cite{AU}). By further observing that $|\nabla d_\Omega|=1$ almost everywhere in $\Omega$, we get that it is actually weakly $p-$superharmonic, for every $1<p<\infty$.
Also observe that by construction, we have $-1 \le u \le 0$.
\par
If we consider the composition $\phi=w_I \circ u$, in light of Lemma \ref{lemma:3.1} and of the properties of $w_I$ contained in Lemma \ref{lem:1d}, we know that $\phi \in W^{1,p}_0(\Omega)$ is a weak positive solution of
\[ 
-\Delta_{p} \phi \ge \frac{1}{r_\Omega^p} \,\phi^{q-1} \quad \text{ in } \Omega. 
\]
By recalling Remark \ref{rem:scalings}, if we define
\[
\widetilde\phi=r_\Omega^\frac{p}{p-q}\,\phi,
\]
then this satisfies
\[
-\Delta_p \widetilde\phi \ge \widetilde\phi^{q-1}, \qquad \mbox{ in } \Omega.
\]
Moreover, we have that both $w_\Omega$ and $\widetilde\phi$ belongs to $W^{1,p}_0(\Omega)$. Thus $(w_\Omega-\widetilde\phi)_+ \in W^{1,p}_0(\Omega)$ and by the comparison principle it holds
\[
w_\Omega \le \widetilde\phi = r_\Omega^{\frac{p}{p-q}}\, w_I\left( \frac{d_{\Omega}}{r_\Omega} - 1 \right), \qquad \text{ a.\,e. in } \Omega,
\]
as desired. 
\vskip.2cm\noindent
{\bf Sharpness.} It is straightfoward to see that the lower bound in \eqref{eq:pointestimate} is sharp. It is sufficient to take $\Omega$ to be any $N-$dimensional open ball and $B_r(x_0)=\Omega$, to get equality in the lower bound.
\par
The upper bound \eqref{eq:pointestimate2} is slightly more complicate: indeed, the function 
\[
w_I\left(\frac{d_{\Omega}(x)}{r_\Omega}-1\right),
\]
``virtually'' coincides with the function $w_{\Omega}$ for the slab $\Omega=\mathbb{R}^{N-1}\times I$. 
However, this choice is not feasible, since $w_\Omega$ is not well-defined in our framework. Indeed, the set $\Omega=\mathbb{R}^{N-1}\times I$ is not $q-$admissible for any $1\le q<p$ and the minimization problem in Definition \ref{defn:w} is not well-posed.
\par
 We go through an approximation argument. For every $n\in\mathbb{N}\setminus\{0\}$, we take
\[
\Omega_n=\left(-\frac{n}{2},\frac{n}{2}\right)^{N-1}\times I,
\]
then from Lemma \ref{lm:basta!} we have that
\begin{equation}
\label{bon}
\lim_{n\to\infty} w_{\Omega_n}(x',x_N)=w_I(x_N),\qquad \mbox{ for a.\,e. } (x',x_N)\in\mathbb{R}^{N-1}\times I.
\end{equation}
On the other hand, by using that $r_{\Omega_n}=1$ for $n\ge 2$ and that 
\[
\lim_{n\to\infty}d_{\Omega_n}(x',x_N)=1-|x_N|,\qquad \mbox{ for } (x',x_N)\in\mathbb{R}^{N-1}\times I
\]
we obtain that
\begin{equation}
\label{bonbon}
\lim_{n\to\infty}
w_I\left(\frac{d_{\Omega_n}(x)}{r_{\Omega_n}}-1\right)=w_I(-|x_N|)=w_I(x_N).
\end{equation}
By comparing \eqref{bon} and \eqref{bonbon}, we get the claimed sharpness.
\end{proof}
As a straightforward application of Theorem \ref{thm:bounds}, we get the following
\begin{cor}[Double-sided $L^\infty$ estimate]
\label{cor:Linfty}
Let $\alpha>0$ and $1\le q<p<\infty$. Let $\Omega \subset \mathbb{R}^N$ be an open bounded convex set. Then we have
\begin{equation}\label{eq:inftyestimate}
w_{B_1,\alpha}(0) \le r_{\Omega}^{-\frac{p}{p-q}} \,\|w_{\Omega,\alpha}\|_{L^{\infty}(\Omega)} \le \left(\alpha\,\left(\frac{2}{\pi_{p,q}}\right)^p\right)^\frac{1}{p-q}\,\left(\frac{q\,p-q+p}{p}\right)^\frac{1}{q},
\end{equation}
and both estimates are sharp.
\end{cor}
\begin{proof}
By still recalling \eqref{walfa}, we can take $\alpha=1$, without loss of generality. Let $B_r(x_0)\subset \Omega$, thus in particular $r\le r_\Omega$, by definition of inradius. By passing to the supremum in \eqref{eq:pointestimate}, we obtain
\[
w_{B_1}(0)=\left\|w_{B_1}\left( \frac{\cdot-x_0}{r} \right)\right\|_{L^\infty(\Omega)}\le r^{-\frac{p}{p-q}} \,\|w_{\Omega,\alpha}\|_{L^{\infty}(\Omega)}.
\]
The first equality follows from the fact that $w_{B_1}$ is a radially symmetric decreasing function.
This well-known property of $w_{B_1}$ can be proved by applying a radially symmetric decreasing rearrangement and then appealing to the so-called {\it P\'olya-Szeg\H{o} principle}, see for example \cite[Theorem 3]{Ka}. By arbitrariness of the radius $r$ in the estimate above, we get the lower bound in \eqref{eq:inftyestimate}.
\par 
As for the upper bound in \eqref{eq:inftyestimate}, we use that $w_I$ is increasing on $(-1,0)$, thus to finish we just need to prove that 
\begin{equation}
\label{volio}
w_I(0)=\left(\frac{q\,p-q+p}{p}\right)^\frac{1}{q}\,\left(\frac{2}{\pi_{p,q}}\right)^\frac{p}{p-q}.
\end{equation}
We now use the identity \eqref{eq:C} and the equation \eqref{eq:mixedproblem} solved by $w_I$, in order to determine $w_I(0)$.
Since $w_I'$ does not vanish in $(-1,0)$, by multiplying equation \eqref{eq:mixedproblem} by $w_I'$, we get  
\[ 
-(|w_I'|^{p-2}\, w_I')' \, w_I'=w_I^{q-1} \, w_I', 
\] 
which can be rewritten as
\[ 
-(p-1)\, \frac{d}{dt} \frac{|w_I'|^p}{p} =\frac{d}{dt} \frac{w_I^q}{q}. 
\]
Upon integrating this on $[t,0]$ and using that $w_I'(0)=0$ , we obtain
$$
(p-1)\,\left(  \frac{|w_I'(t)|^p}{p} \right) = \frac{(w_I(0))^q}{q} -\frac{(w_I(t))^q}{q} .
$$
which implies
\begin{equation}\label{eq:E}
(w_I'(t))^p =  \frac{p}{q\,(p-1)} ( (w_I(0))^q-(w_I(t))^q ). 
\end{equation}
Finally, by using again the equation, the evenness of $w_I$ and \eqref{eq:E}, we obtain
\[
\int_{-1}^{0} |w_I(t)|^q \, dt = \int_{-1}^{0} |w_I'(t)|^p \, dt = \frac{p}{q\,(p-1)}\, (w_I(0))^q - \frac{p}{q\,(p-1)}\, \int_{-1}^{0} |w_I(t)|^q \, dt,
\]
hence 
\[ 
w_I(0)=\left( \frac{q\,p-q+p}{p} \right)^\frac{1}{q}\, \left( \int_{-1}^{0}|w_I(t)|^q\,dt \right)^\frac{1}{q}. 
\]
From this and \eqref{eq:C}, we finally get \eqref{volio}. 
\par The sharpness of our $L^\infty$ estimate is now a straightforward consequence of the sharpness of \eqref{eq:pointestimate}.
\end{proof}

\begin{rem}[Universal $L^\infty$ estimate]
\label{rem:universal}
By combining Corollaries \ref{cor:hierarchy} and \ref{cor:Linfty}, we get in particular that in an open bounded convex set $\Omega\subset\mathbb{R}^N$ we have
\[
r_{\Omega}^{-\frac{p}{p-q}} \,\|v\|_{L^{\infty}(\Omega)} \le \left(\alpha\,\left(\frac{2}{\pi_{p,q}}\right)^p\right)^\frac{1}{p-q}\,\left(\frac{q\,p-q+p}{p}\right)^\frac{1}{q},
\]
for {\it every} solution $v\in W^{1,p}_0(\Omega)$ of \eqref{LE}. 
\end{rem}

\subsection{Localization of maximum points}
Theorem \ref{thm:bounds} gives quite a precise description of $w_{\Omega,\alpha}$, in terms of geometric quantities. It is thus possible to use this description to give a simple localization estimate for the maximum points of $w_{\Omega,\alpha}$. This is the content of the following

\begin{cor}
\label{cor:localization}
Let $\alpha>0$  and $1\le q<p<\infty$. Let $\Omega \subset \mathbb{R}^N$ be an open bounded convex set. For every maximum point $x_0\in \Omega$ of $w_{\Omega,\alpha}$, we have
\[
d_\Omega(x_0)\ge C_{N,p,q}\,r_\Omega,
\] 
where the constant $0<C_{N,p,q}<1$ is defined by 
\[
C_{N,p,q}:=\left(\frac{q\,(p-1)}{p}\right)^\frac{1}{p}\,\int_0^{w_{B_1}(0)} \frac{1}{\left((w_I(0))^q-\tau^q\right)^\frac{1}{p}}\,d\tau,
\]
and $w_I(0)$ has been evaluated in \eqref{volio}.
\end{cor}
\begin{proof}
Again by \eqref{walfa}, we see that the location of maximum points is independent of $\alpha>0$. 
Thus we can take $\alpha=1$. We also observe that by standard results from Elliptic Regularity, we know that $w_\Omega$ is continuous on $\overline\Omega$ (see for example \cite[Theorem 7.8]{Gi}). Thus, it does admit maximum points on $\Omega$. If $x_0\in \Omega$ is such a maximum point, we get from \eqref{eq:pointestimate} and \eqref{eq:pointestimate2}
\[
w_{B_1}(0) \le r_\Omega^{-\frac{p}{p-q}}\, \,w_{\Omega}(x_0) \le w_{I}\left(\frac{d_{\Omega}(x_0)}{r_\Omega}-1\right).
\]
This in particular entails that
\begin{equation}
\label{estate}
w_{B_1}(0)\le w_{I}\left(\frac{d_{\Omega}(x_0)}{r_\Omega}-1\right).
\end{equation}
We recall that $d_\Omega/r_\Omega -1\le 0$ and observe that 
\[
w_I:(-1,0]\to (0,w_I(0)] 
\]
is increasing, thus invertible. 
By taking the inverse function of $w_I$, we get
\begin{equation}
\label{hotspot}
w_{I}^{-1}(w_{B_1}(0))+1\le \frac{d_{\Omega}(x_0)}{r_\Omega}.
\end{equation}
We now observe that for every $y\in (0,w_I(0)]$ we can write
\[
w_I^{-1}(y)=w_I^{-1}(0)+\int_0^y \frac{1}{w_I'(w_I^{-1}(\tau))}\,d\tau=-1+\int_0^y \frac{1}{w_I'(w_I^{-1}(\tau))}\,d\tau.
\]
We can use \eqref{eq:E} to evalute the derivative inside the integral. This gives
\[
w_I^{-1}(y)+1=\left(\frac{q\,(p-1)}{p}\right)^\frac{1}{p}\,\int_0^y \frac{1}{\left(w_I(0)^q-\tau^q\right)^\frac{1}{p}}\,d\tau.
\]
By recalling the definition of $C_{N,p,q}$, from \eqref{hotspot} we get the desired conclusion.
\end{proof}
\begin{rem}
A result similar to the previous one has been obtained by Magnanini and Poggesi in \cite{MP}. We refer in particular to \cite[Remark 4.8]{MP}, even if the estimate is not explicitly stated for our equation. Their proof is different from ours, it is based on obtaining a refined gradient bound for the solution $w_\Omega$. Their method of proof is inspired by the so-called {\it $P-$functions method}, introduced by Payne in \cite{Pa}.  
\par
For the case $q=1$, by using the explicit expression of $w_{B_1(0)}$ and $w_I$ (see Remark \ref{rem:case1}), directly from \eqref{estate} it is not difficult to get the expression 
\[
C_{N,p,1}=1-(1-N^{-\frac{1}{p-1}})^\frac{p-1}{p}.
\]
This is worse than the constant $N^{-1/p}$ obtained by Magnanini and Poggesi, see \cite[Corollary 4.3]{MP}. In both cases, we observe that such a constant tends to $1$, as $p$ goes to $\infty$. Accordingly, the maximum points of $w_\Omega$ get closer and closer to the maximum points of the distance function $d_\Omega$. 
\end{rem}
 
\subsection{Generalized principal frequencies}

By using the upper bound of Theorem \ref{thm:bounds}, we can prove the following sharp geometric estimate for $\lambda_{p,q}$, in convex sets. This generalizes to the case $p\not =2$ the result of \cite[Theorem 1.1]{BM}. We point out that the proof here is slightly different from that of \cite{BM}, since we now rely on Theorem \ref{thm:bounds}, which is turn follows from the comparison principle.
\begin{thm}[Hersch-Protter--type inequality]
\label{thm:HP}
 Let  $1 \le q < p<\infty$. Let $\Omega \subset \mathbb{R}^N$ be an open bounded convex set. Then the following lower bound holds
\begin{equation}
\label{eq:upper bound inr}
\lambda_{p,q}(\Omega)\,|\Omega|^{\frac{p-q}{q}} \ge \left( \frac{\pi_{p,q}}{2} \right)^p\,\frac{1}{r_\Omega^p}.
\end{equation}
Moreover, the estimate is sharp.
\end{thm}
\begin{proof}
We take $w_{\Omega,\alpha}\in W^{1,p}_0(\Omega)$ the unique positive solution of \eqref{LE} with homogeneous Dirichlet boundary conditions, corresponding to the choice $\alpha=\lambda_{p,q}(\Omega)$. By uniqueness, this must coincide with the positive minimizer of
\[
\lambda_{p,q}(\Omega)=\min_{\psi\in W^{1,p}_0(\Omega)}\left\{\int_\Omega |\nabla \psi|^p\,dx\, :\, \int_\Omega |\psi|^q\,dx=1\right\},
\]
which is a positive solution of the same boundary value problem, by optimality.
Thus in particular we have 
\begin{equation}
\label{norma1}
\int_\Omega |w_{\Omega,\alpha}|^q\,dx=1.
\end{equation}
We also observe that from Theorem \ref{thm:bounds}, we get
\[
w_{\Omega,\alpha}(x) \le r_\Omega^{\frac{p}{p-q}}\, w_{I,\alpha}\left(\frac{d_{\Omega}(x)}{r_\Omega}-1\right),\qquad \mbox{ in }\Omega.
\]
We raise to the power $q$ and integrate over $\Omega$. By taking \eqref{norma1} into account and using \eqref{walfa} with $\alpha=\lambda_{p,q}(\Omega)$, this yields
\begin{equation}
\label{siparte}
1\le \left(\lambda_{p,q}(\Omega)\,r_\Omega^p\right)^{\frac{q}{p-q}}\, \int_\Omega \left[w_I\left(\frac{d_{\Omega}(x)}{r_\Omega}-1\right)\right]^q\,dx.
\end{equation}
In order to conclude the proof, we need to extract the geometrical content from the last integral.
\par
By using the {\it Coarea Formula} and the fact that $|\nabla d_\Omega|=1$ almost everywhere in $\Omega$, we can write
\[
\int_\Omega \left[w_I\left(\frac{d_{\Omega}(x)}{r_\Omega}-1\right)\right]^q\,dx = \int_{0}^{r_\Omega} \left[w_I\left(\frac{t}{r_\Omega}-1\right) \right]^{q}\, P(t)  \, dt, 
\]
where we set 
\[
P(t)=\mathcal{H}^{N-1}(\partial\Omega_t) \qquad \mbox{ and }\qquad \Omega_t=\Big\{x\in \Omega\, :\, d_\Omega(x)>t\Big\}.
\] 
We introduce the function
\[
\xi(t)=\int_{0}^t \left[w_I\left(\frac{\tau}{r_\Omega}-1\right) \right]^{q}\, d\tau.
\]
It is not difficult to see that $t\mapsto \xi(t)/t$ is monotone increasing, while $t\mapsto P(t)$ is decreasing by convexity of $\Omega$ (see for example \cite[Lemma 2.2.2]{BB}). By appealing to \cite[Lemma A.1]{BM}, we get
\[
\int_{0}^{r_\Omega} \left[w_I\left(\frac{t}{r_\Omega}-1\right) \right]^{q}\, P(t)  \, dt\le \frac{\xi(r_\Omega)}{r_\Omega}\,\int_0^{r_\Omega} P(t)\,dt.
\] 
By recalling the definition of $\xi$ and using again Coarea Formula, this is the same as
\[
\int_{0}^{r_\Omega} \left[w_I\left(\frac{t}{r_\Omega}-1\right) \right]^{q}\, P(t)  \, dt\le \frac{|\Omega|}{r_\Omega}\,\int_{0}^{r_\Omega} \left[w_I\left(\frac{t}{r_\Omega}-1\right) \right]^{q}\,dt.
\]
Finally,  
by making the change of variable $s=t/r_\Omega-1$, we obtain
\[
\int_\Omega \left[w_I\left(\frac{d_\Omega(x)}{r_\Omega}-1\right)\right]^q\,dx \le |\Omega|\, \int_{-1}^{0} w_I(s)^{q} \, ds.
\]
By inserting this estimate into \eqref{siparte} and recalling \eqref{eq:C}, we eventually conclude the proof of the inequality.
\vskip.2cm\noindent	
As in the case $p=2$,  we show that inequality \eqref{eq:upper bound inr} is asimptotically sharp for the slab-type sequence 
\[
\Omega_L = \left( -\frac{L}{2}, \frac{L}{2} \right)^{N-1} \times I \subset \mathbb{R}^N.
\]
Indeed, for $L>1$ we have that 
\begin{equation}
\label{inraggio}
r_{\Omega_L}= 1 \qquad \text{ hence }\qquad \left( \frac{\pi_{p,q}}{2} \right)^p\,\frac{1}{r_\Omega^p}=\left( \frac{\pi_{p,q}}{2} \right)^p. 
\end{equation}
In order to estimate $\lambda_{p,q}(\Omega_L)$, we use that (see \cite[Main Theorem]{Bra1})	
\[ 
\lambda_{p,q}(\Omega_L) \le \left( \frac{\pi_{p,q}}{2} \right)^p\, \left( \frac{P(\Omega_L)}{|\Omega_L|^{1-\frac{1}{p}+\frac{1}{q}}} \right)^p.
\]
By joining this estimate and \eqref{eq:upper bound inr}, we get
\[
1\le \left( \frac{2}{\pi_{p,q}} \right)^p\,r_{\Omega_L}^p\,|\Omega_L|^\frac{p-q}{q}\,\lambda_{p,q}(\Omega_L)\le r_{\Omega_L}^p\,\left(\frac{P(\Omega_L)}{|\Omega_L|}\right)^p.
\]
If we now recall \eqref{inraggio} and use that 
\[
P(\Omega_L)\sim 2 \,L^{N-1},\qquad |\Omega_L|=2\,L^{N-1},\qquad \mbox{ as } L\to +\infty,
\]
we get
\[
\lim_{L\to +\infty}\left[\left( \frac{2}{\pi_{p,q}} \right)^p\,r_{\Omega_L}^p\,|\Omega_L|^\frac{p-q}{q}\,\lambda_{p,q}(\Omega_L)\right]=1,
\]
which proves the claimed sharpness of the estimate.
\end{proof}

\begin{rem}[More general sets]
Apart for the simple lower bound \eqref{eq:pointestimate}, all the results of Section \ref{sec:5} have been proved under the assumption that $\Omega$ is convex.
Actually, in the proof of Theorem \ref{thm:bounds}, convexity was used in the proof of the upper bound \eqref{eq:pointestimate2} only to assure that the distance function $d_\Omega$ was weakly superharmonic. 
Thus the upper bound \eqref{eq:pointestimate2} (and consequently all its consequences in Section \ref{sec:5}) continues to hold for all sets such that
\begin{equation}
\label{wsh}
-\Delta d_{\Omega} \ge 0,\qquad \mbox{ in }\Omega,
\end{equation}
in weak sense. For completeness, we recall that condition \eqref{wsh} is equivalent to require that $\Omega$ is convex in dimension $N=2$, but it is otherwise a weaker condition for $N\ge 3$, see \cite{AU}.
\end{rem}

%%%%%%%%%%%%%%%%%%%%%%%%%%%%%%%%%%%%%%%%%%%%%%%%%%%%%%%%%%

\appendix
%\section{Appendix}

\section{Quantified convexity of power functions}
\label{app:A}

\begin{lemma}
\label{lm:lambdaconvexgen}
Let $r\ge 2$. For every $z,w\in\mathbb{R}^N$ and every $t\in[0,1]$ we have
\[
t\,|z|^r+(1-t)\, |w|^r\ge |t\, z+(1-t)\, w|^r+C\, t\,(1-t)\, |z-w|^r,
\]
where $C=C(r)>0$.
\end{lemma}
\begin{proof}
For simplicity, we set $F(z)=|z|^r$. For $t=0$ or $t=1$ there is nothing to prove, so let us assume $0<t<1$. From \cite[Lemma 4.2, equation (4.3)]{Lin}, we know that there exists $C_r>0$ such that
\begin{equation}
\label{lambdaconvex2bis}
F(\xi)\ge F(\zeta)+\langle \nabla F(\zeta),\xi-\zeta\rangle+C_r\,|\xi-\zeta|^r,\qquad \mbox{ for every } \xi,\zeta\in \mathbb{R}^N.
\end{equation}
We use \eqref{lambdaconvex2bis} with 
\[
\xi=z\qquad \mbox{ and }\qquad \zeta=t\,z+(1-t)\,w.
\]
We obtain
 \begin{equation}
\label{z}
\begin{split}
F(z)&\ge F(t\,z+(1-t)\,w)+(1-t)\,\langle \nabla F(t\,z+(1-t)\,w),z-w\rangle+C_r\, (1-t)^r\,|z-w|^r.
\end{split}
\end{equation}
Similary, we use \eqref{lambdaconvex2bis} with 
\[
\xi=w\qquad \mbox{ and }\qquad \zeta=t\,z+(1-t)\,w
\]
This now yields
\begin{equation}
\label{w}
F(w)\ge F(t\,z+(1-t)\,w)+t\,\langle \nabla F(t\,z+(1-t)\,w),w-z\rangle+C_r\, t^r\,|z-w|^r.
\end{equation}
We multiply \eqref{z} by $t$, then multiply \eqref{w} by $1-t$ and sum up. The outcome is the following
\[
\begin{split}
(1-t)\,F(w)+t\, F(z)&\ge F(t\,z+(1-t)\,w)\\
&+C_r\, \left[(1-t)^{r-1}+t^{r-1}\right]\,t\,(1-t)\, |z-w|^r.
\end{split}
\]
By using convexity of the function $\tau\mapsto\tau^{r-1}$, we get
\[
(1-t)^{r-1}+t^{r-1}\ge 2^{2-r},
\] 
and thus the conclusion.
\end{proof}
\begin{rem}
We observe that the extra term $C_r\,|\xi-\zeta|^r$ in \eqref{lambdaconvex2bis} permits to prove an improved version of the classical Jensen inequality for the convex function $F(z)=|z|^r$, containing a suitable remainder term. A general class of functions which satisfy this kind of stronger Jensen's inequality  is widely studied  in \cite{AJS}.  
We owe this remark and reference \cite{AJS} to the kind courtesy of an anonymous referee.
\end{rem}
\begin{lemma}
\label{lm:lambdaconvexgensub}
Let $1<r<2$. For every $z,w\in\mathbb{R}^N$ and every $t\in[0,1]$ we have
\[
t\,|z|^r+(1-t)\, |w|^r\ge |t\, z+(1-t)\, w|^r+C\, t\,(1-t)\,\left(|z|^2+|w|^2\right)^\frac{r-2}{2}\, |z-w|^2,
\]
where $C=C(r)>0$.
\end{lemma}
\begin{proof}
The proof is the same as that of Lemma \ref{lm:lambdaconvexgen}. It is sufficient to use this time \cite[Lemma 4.2, equation (4.4)]{Lin}. We leave the details to the reader.
\end{proof}

\section{Asymptotics of the positive solution in a slab-type sequence} 
\label{app:C}

We still use the notation $w_\Omega$ of Definition \ref{defn:w}.
For every $L>0$, we indicate by 
\begin{equation}
\label{slaba}
\Omega_L=\left(-\frac{L}{2},\frac{L}{2}\right)^{N-1}\times I,
\end{equation}
where we recall that $I=(-1,1)$. We then have the following convergence result.
\begin{lemma}
\label{lm:basta!}
Let $1\le q<p<\infty$, we define the function
\[
U_\infty(x',x_N)=w_I(x_N),\qquad \mbox{ for } x'\in \mathbb{R}^{N-1}, x_N\in I.
\] 
Then for every $L_0>0$ we have 
\[
\lim_{n\to\infty} \|w_{\Omega_n}-U_\infty\|_{L^p(\Omega_{L_0})}=0.
\]
\end{lemma}
\begin{proof}
We will adapt to our nonlinear situation a related argument from \cite[Lemma 7.2]{BFR}, for the case $p=2$ and $q=1$.
\par
We extend all functions $w_{\Omega_L}$ to the whole slab $\mathbb{R}^{N-1}\times I$, by putting them constantly equal to $0$ outside $\Omega_L$. Observe that, by Elliptic Regularity (see for example \cite[Theorem 7.8]{Gi}), we know that $w_{\Omega_L}$ is H\"older continuous on $\overline{\Omega_L}$ and thus it takes the homogeneous Dirichlet boundary condition in classical pointwise sense. Accordingly, the extended functions are H\"older continuous on $\mathbb{R}^{N-1}\times I$. 
\par
We first observe that  
\[
w_{\Omega_{L_2}}\ge w_{\Omega_{L_1}},\qquad \mbox{ for } L_2\ge L_1,
\]
by the comparison principle of Theorem \ref{thm:comparison}, thanks to the fact that $\Omega_{L_2}\supset \Omega_{L_1}$. Thus, we get that $\{w_{\Omega_L}\}_{L>0}$ is a family of monotone increasing continuous functions.
Moreover, we have the uniform upper bound
\begin{equation}
\label{upperbound}
w_{\Omega_L}\le w_{I}(0),\qquad \mbox{ in } \mathbb{R}^{N-1}\times I,
\end{equation}
thanks to \eqref{eq:pointestimate}.
Then, the pointwise limit
\begin{equation}
\label{monot}
U(x)=\lim_{L\to +\infty} w_{\Omega_L}(x),\qquad \mbox{ for } x\in  \mathbb{R}^{N-1}\times I,
\end{equation}
is well-defined. Observe that this is a bounded function, which still satisfies \eqref{upperbound}. We also notice that, if we fix $L_0>0$ as in the statement, we have 
\begin{equation}
\label{LpU}
\lim_{L\to+\infty} \int_{\Omega_{4L_0}} |w_{\Omega_L}-U|^p\,dx=\lim_{n\to\infty} \int_{\Omega_{4L_0}} (U-w_{\Omega_L})^p\,dx=0,
\end{equation}
thanks to the Monotone Convergence Theorem. We devote the rest of the proof to show that 
\begin{equation}
\label{uguali1d}
U=U_\infty,\qquad \mbox{ in } \mathbb{R}^{N-1}\times I.
\end{equation}
Let us now work with the sequence $\{w_{\Omega_n}\}_{n\ge 4L_0}$,  where $\Omega_n$ is defined by \eqref{slaba}, with the choice $L=n$. We take $\eta\in C^{\infty}(\overline{\Omega_{4L_0}})$ a cut-off function such that
\[
0\le \eta\le 1,\qquad \eta\equiv 1\quad \mbox{ on } \overline{\Omega_{2L_0}} ,\qquad \|\nabla \eta\|_{L^\infty}\le \frac{C}{L_0},
\]
and 
\[
\eta\equiv 0\quad \mbox{ on } \Omega_{4L_0}\setminus\Omega_{3L_0}. 
\]
Then we insert the test function $\psi=\,\eta^p\,w_{\Omega_n}$ in the weak formulation of \eqref{LE} for $w_{\Omega_n}$. We get
\[
\begin{split}
\int_{\Omega_n} |\nabla w_{\Omega_n}|^p\,\eta^p\,dx+p\,\int_{\Omega_n} \langle |\nabla w_{\Omega_n}|^{p-2}\,\nabla w_{\Omega_n},\nabla \eta\rangle\,\eta^{p-1}\,w_{\Omega_n}\,dx=\int_{\Omega_n} w_{\Omega_n}^q\,\eta^p\,dx.
\end{split}
\]
By Young's inequality, for every $\delta>0$ we have
\[
\begin{split}
p\, \langle |\nabla w_{\Omega_n}|^{p-2}\,\nabla w_{\Omega_n},\nabla \eta\rangle\,\eta^{p-1}\,w_{\Omega_n}&\ge -\delta\,(p-1)\,|\nabla w_{\Omega_n}|^p\,\eta^p\\
&-\delta^{-(p-1)}\,w_{\Omega_n}^p\,|\nabla \eta|^p.
\end{split}
\]
By choosing $\delta=1/(2\,p-2)$, we then obtain the Caccioppoli-type inequality
\[
\frac{1}{2}\,\int_{\Omega_n} |\nabla w_{\Omega_n}|^p\,\eta^p\,dx\le (2\,p-2)^{p-1} \,\int_{\Omega_n} |\nabla \eta|^p\,w_{\Omega_n}^p\,dx+\int_{\Omega_n} w_{\Omega_n}^q\,\eta^p\,dx.
\]
By using this estimate, the properties of $\eta$ and the upper bound \eqref{upperbound}, we then obtain in particular
\begin{equation}
\label{boundgrad}
\int_{\Omega_{2L_0}} |\nabla w_{\Omega_n}|^p\,dx\le C,\qquad \mbox{ for every } n\ge 4\,L_0.
\end{equation}
for some uniform constant $C>0$.
This implies that the sequence $\{w_{\Omega_n}\}_{n\ge 4L_0}$ is bounded in $W^{1,p}(\Omega_{2L_0})$. Thus it weakly converges to a function $u\in W^{1,p}(\Omega_{2L_0})$: by uniqueness of the limit, we must have $u=U$. This in particular implies that $U$ belongs to $W^{1,p}(\Omega_{2L_0})$. By also using the compactness of the trace embedding (see \cite[Corollary 18.4]{Leonibook})
\[
W^{1,p}(\Omega_{2L_0})\hookrightarrow L^p(\partial \Omega_{2L_0}),
\]
and the boundary condition
\[
w_{\Omega_n}=0,\qquad \mbox{ on } \left(-L_0,L_0\right)^{N-1}\times\{-1,\,1\},
\]
we get that the trace of $U$ must have the same property.
\vskip.2cm\noindent
We claim that $U$ weakly solves the Lane-Emden equation \eqref{LE} in $\Omega_{L_0}$. We take $\zeta\in C^{\infty}(\overline{\Omega_{3L_0}})$ a cut-off function such that
\[
0\le \zeta\le 1,\qquad \zeta\equiv 1\quad \mbox{ on } \overline{\Omega_{L_0}} ,\qquad \|\nabla \zeta\|_{L^\infty}\le \frac{C}{L_0},
\]
and 
\[
\zeta\equiv 0\quad \mbox{ on } \Omega_{3L_0}\setminus\Omega_{2L_0}. 
\]
Then we have 
\begin{equation}
\label{pezzetti}
\begin{split}
\int_{\Omega_{2L_0}} \langle |\nabla w_{\Omega_n}|^{p-2}\,\nabla w_{\Omega_n},\nabla w_{\Omega_n}-\nabla U\rangle\,\zeta\,dx&=\int_{\Omega_{2L_0}} \langle |\nabla w_{\Omega_n}|^{p-2}\,\nabla w_{\Omega_n},\nabla ((w_{\Omega_n}-U)\,\zeta)\rangle\,dx\\
&-\int_{\Omega_{2L_0}} \langle |\nabla w_{\Omega_n}|^{p-2}\,\nabla w_{\Omega_n},\nabla \zeta\rangle\,(w_{\Omega_n}-U)\,dx\\
&=\int_{\Omega_{2L_0}} w_{\Omega_n}^{q-1}\,(w_{\Omega_n}-U)\,\zeta\,dx\\
&-\int_{\Omega_{2L_0}} \langle |\nabla w_{\Omega_n}|^{p-2}\,\nabla w_{\Omega_n},\nabla \zeta\rangle\,(w_{\Omega_n}-U)\,dx,
\end{split}
\end{equation}
where we used the equation for $w_{\Omega_n}$, tested against the function $\psi=\zeta\,(w_{\Omega_n}-U)$. Indeed, this is a  feasible test function, thanks to the condition on the trace of $w_{\Omega_n}-U$. We now observe that 
\[
\lim_{n\to\infty} \int_{\Omega_{2L_0}} w_{\Omega_n}^{q-1}\,(w_{\Omega_n}-U)\,\zeta\,dx=0,
\]
thanks to the uniform bound \eqref{upperbound} and to the strong convergence \eqref{LpU}. Moreover, by using \eqref{boundgrad} and again \eqref{LpU}, we also get
\[
\lim_{n\to\infty} \int_{\Omega_{2L_0}} \langle |\nabla w_{\Omega_n}|^{p-2}\,\nabla w_{\Omega_n},\nabla \zeta\rangle\,(w_{\Omega_n}-U)\,dx=0.
\] 
On account of \eqref{pezzetti}, these yield
\[
\lim_{n\to\infty}\int_{\Omega_{2L_0}} \langle |\nabla w_{\Omega_n}|^{p-2}\,\nabla w_{\Omega_n},\nabla w_{\Omega_n}-\nabla U\rangle\,\zeta\,dx=0.
\]
Moreover, we also have 
\[
\lim_{n\to\infty}\int_{\Omega_{2L_0}} \langle |\nabla U|^{p-2}\,\nabla U,\nabla w_{\Omega_n}-\nabla U\rangle\,\zeta\,dx=0,
\]
thanks to the weak convergence in $W^{1,p}(\Omega_{2L_0})$ of $w_{\Omega_n}$. By subtracting the last two equations in display, we get
\[
\lim_{n\to\infty} \int_{\Omega_{2L_0}} \langle |\nabla w_{\Omega_n}|^{p-2}\,\nabla w_{\Omega_n}-|\nabla U|^{p-2}\,\nabla U,\nabla w_{\Omega_n}-\nabla U\rangle\,\zeta\,dx=0,
\]
as well. By recalling that $\zeta$ is constantly equal to $1$ on $\overline{\Omega_{L_0}}$ and that the integrand is non-negative, we get in particular
\[
\lim_{n\to\infty} \int_{\Omega_{L_0}} \langle |\nabla w_{\Omega_n}|^{p-2}\,\nabla w_{\Omega_n}-|\nabla U|^{p-2}\,\nabla U,\nabla w_{\Omega_n}-\nabla U\rangle\,dx=0.
\]
For $p\ge 2$, by recalling the inequality (see \cite[Section 10, equation (I)]{Lin2})
\[
\langle |z|^{p-2}\,z-|w|^{p-2}\,w,z-w\rangle\ge 2^{2-p}\,|z-w|^p,\qquad \mbox{ for every } z,w\in\mathbb{R}^N,
\]
we immediately obtain
\begin{equation}
\label{vamos!}
\lim_{n\to\infty} \|\nabla w_{\Omega_n}-\nabla U\|_{L^p(\Omega_{L_0};\mathbb{R}^N)}=0.
\end{equation}
For $1<p<2$, it is slightly more complicate: we need to use the inequality (see \cite[Section 10, equation (VII)]{Lin2})
\[
\langle |z|^{p-2}\,z-|w|^{p-2}\,w,z-w\rangle\ge (p-1)\,|z-w|^2\,(1+|z|^2+|w|^2)^\frac{p-2}{2},\qquad \mbox{ for every } z,w\in\mathbb{R}^N.
\]
This permits to infer, thanks to H\"older's inequality, that we have
\[
\begin{split}
\int_{\Omega_{L_0}} |\nabla w_{\Omega_n}-\nabla U|^p\,dx&\le \left(\int_{\Omega_{L_0}} |\nabla w_{\Omega_n}-\nabla U|^2\,(1+|\nabla w_{\Omega_n}|^2+|\nabla U|^2)^\frac{p-2}{2}\,dx\right)^\frac{p}{2}\\
&\times \left(\int_{\Omega_{L_0}}(1+|\nabla w_{\Omega_n}|^2+|\nabla U|^2)^\frac{p}{2}\,dx\right)^\frac{2-p}{2}\\
&\le \left(\frac{1}{p-1}\,\int_{\Omega_{L_0}} \langle |\nabla w_{\Omega_n}|^{p-2}\,\nabla w_{\Omega_n}-|\nabla U|^{p-2}\,\nabla U,\nabla w_{\Omega_n}-\nabla U\rangle\,dx\right)^\frac{p}{2}\\
&\times \left(\int_{\Omega_{L_0}}(1+|\nabla w_{\Omega_n}|^2+|\nabla U|^2)^\frac{p}{2}\,dx\right)^\frac{2-p}{2}\\
\end{split}
\]
By using that the last integral is uniformly bounded thanks to \eqref{boundgrad}, while the other one converges to $0$, we get \eqref{vamos!} for $1<p<2$, as well.
\par
Thanks to the strong convergence \eqref{vamos!}, we can now pass to the limit in 
\[
\int_{\Omega_{L_0}} \langle |\nabla w_{\Omega_n}|^{p-2}\,\nabla w_{\Omega_n},\nabla \psi\rangle\,dx=\int_{\Omega_{L_0}} w_{\Omega_n}^{q-1}\,\psi\,dx,\qquad \mbox{ for every } \psi\in C^\infty_0(\Omega_{L_0}),
\]
and obtain
\[
\int_{\Omega_{L_0}} \langle |\nabla U|^{p-2}\,\nabla U,\nabla \psi\rangle\,dx=\int_{\Omega_{L_0}} U^{q-1}\,\psi\,dx,\qquad \mbox{ for every } \psi\in C^\infty_0(\Omega_{L_0}),
\]
i.\,e. $U$ is a solution of the Lane-Emden equation in $\Omega_{L_0}$, as claimed.
\vskip.2cm\noindent
Next, we claim that $U$ actually does not depend on the variable $x'$, but only on $x_N$. To prove this, let us fix two points 
\[
X=(x'_0,x_N),\, Y=(x_1',x_N)\in \mathbb{R}^{N-1}\times I,\qquad \mbox{ with } x_0'\not=x_1'.
\] 
By definition, we have 
\[
U(X)=\lim_{n\to\infty} w_{\Omega_n}(X)\qquad \mbox{ and }\qquad U(Y)=\lim_{n\to\infty} w_{\Omega_n}(Y).
\]
We introduce the ``horizontally'' translated set $\widetilde\Omega_n=\Omega_n-X+Y$ and observe that we clearly have
\[
w_{\widetilde \Omega_n}(x)=w_{\Omega_n}(x+X-Y),\qquad \mbox{ for } x\in \widetilde \Omega_n.
\]
We notice that by construction, we have that there exists $n_0=n_0(|X-Y|)\in \mathbb{N}$ such that 
\[
\Omega_{\frac{n}{4}}\subset \widetilde{\Omega}_n\subset \Omega_{4\,n},\qquad \mbox{ for every } n\ge n_0.
\]
Again by the comparison principle of Theorem \ref{thm:comparison}, we know that 
\[
w_{\Omega_\frac{n}{4}}\le w_{\widetilde \Omega_n}\le w_{\Omega_{4\,n}},
\]
and moreover, we have 
\[
\lim_{n\to\infty} w_{\Omega_\frac{n}{4}}(x)=\lim_{n\to\infty}w_{\Omega_{4\,n}}(x)=U(x),
\]
thanks to \eqref{monot}. This in turn implies that 
\[
\lim_{n\to\infty} w_{\widetilde \Omega_n}(x)=U(x),
\]
as well.
We then obtain
\[
\begin{split}
U(X)=\lim_{n\to\infty} w_{\Omega_n}(X)=\lim_{n\to\infty} w_{\widetilde \Omega_n}(Y)=U(Y),
\end{split}
\]
as desired.
\vskip.2cm\noindent
By resuming, we get that $U$ is a positive weak solution of 
\[
\left\{\begin{array}{rcll}
-\Delta_p U&=&U^{q-1},& \mbox{ in } \Omega_{L_0},\\
&&&\\
U&=&0,& \mbox{ on } \displaystyle\left(-\frac{L_0}{2},\frac{L_0}{2}\right)^{N-1}\times\{-1,\,1\},
\end{array}
\right.
\]
which does not depend on the variable $x'$. In particular, we have that $x_N\mapsto U(x',x_N)$ is a weak solution of the same one-dimensional problem solved by $w_I$, as well.
By uniqueness of the solution and arbitrariness of $L_0$, we finally get \eqref{uguali1d}.
This concludes the proof.
\end{proof}
\begin{rem}
The $L^p$ convergence of the previous result can actually be upgraded to a uniform convergence. It is sufficient to observe that $U_\infty$ is continuous, that $\{w_{\Omega_n}\}_{n\in\mathbb{N}}$ is a sequence of monotone increasing continuous functions and then use Dini's Theorem. We leave the details to the reader.
\end{rem}

\end{document}